\documentclass[11pt]{amsart}
\usepackage[margin=1.31in]{geometry}

\usepackage{amsmath}%
\usepackage{amsfonts}%
\usepackage{amssymb}%
\usepackage{amsthm}
\usepackage{graphicx}
\usepackage{dsfont}
\usepackage{pdfcomment}
\usepackage{tikz}
\usepackage{tikz-3dplot}
\usepackage{subfigure}
\usepackage{soul}
\usepackage{enumerate}
\usepackage{cancel}

\usepackage{hyperref}
\usepackage{cleveref}
\usepackage{esvect}

\usetikzlibrary{decorations.markings}
\newtheorem{theorem}{Theorem}[section]

\newtheorem{corollary}[theorem]{Corollary}

\newtheorem{definition}[theorem]{Definition}
\newtheorem{example}{Example}
\allowdisplaybreaks

\newtheorem{lemma}[theorem]{Lemma}

\newtheorem{proposition}[theorem]{Proposition}
\newtheorem{remark}[theorem]{Remark}

\DeclareMathOperator{\Arg}{Arg}

\DeclareMathOperator{\Area}{Area}

\DeclareMathOperator{\Real}{Re}
\DeclareMathOperator{\Imaginary}{Im}

\DeclareMathOperator{\sign}{sign}

\DeclareMathOperator{\trace}{trace}
\DeclareMathOperator{\area}{Area}
\DeclareMathOperator{\Herm}{Herm}

\renewcommand{\Re}{\Real}
\renewcommand{\Im}{\Imaginary}

\title[Osculating M\"{o}bius transformations between circle patterns]{CMC-1 surfaces via osculating M\"{o}bius transformations between circle patterns}
\author{Wai Yeung Lam}
\thanks{
This work was partially supported by the ANR/FNR project SoS, INTER/ANR/16/11554412/SoS,
	ANR-17-CE40-0033, FNR grant CoSH O20/14766753.}

\address{Department of Mathematics, University of Luxembourg, Maison du nombre, 6 avenue de la Fonte, L-4364 Esch-sur-Alzette, Luxembourg}

\email{wai.lam@uni.lu}

\begin{document}
	
	\begin{abstract}
		Given two circle patterns of the same combinatorics in the plane, the M\"{o}bius transformations mapping circumdisks of one to the other induces a $PSL(2,\mathbb{C})$-valued function on the dual graph. Such a function plays the role of an osculating M\"{o}bius transformation and induces a realization of the dual graph in hyperbolic space. We characterize the realizations and obtain a one-to-one correspondence in the cases that the two circle patterns share the same discrete conformal structure. These correspondences are analogous to the Weierstrass representation for surfaces with constant mean curvature $H\equiv 1$ in hyperbolic space. We further establish convergence on triangular lattices.
	\end{abstract}
	
		\maketitle
		
		Discrete differential geometry concerns structure-preserving discretizations in differential geometry. Its goal is to establish a discrete theory with rich mathematical structures such that the smooth theory arises in the limit of refinement. It has stimulated applications in computational architecture and computer graphics. 
		
		An example in discrete conformal geometry is William Thurston's circle packing \cite{Stephenson2005}. In the classical theory, holomorphic functions are conformal, mapping infinitesimal circles to themselves. Instead of infinitesimal size, a circle packing is a configuration of finite-size circles where certain pairs are mutually tangent. Thurston proposed regarding a map induced from two circle packings with the same tangency pattern as a discrete holomorphic function. A discrete analogue of the Riemann mapping follows from Koebe-Andreev-Thurston theorem. Rodin and Sullivan \cite{Rodin1987} showed that it converges to the classical Riemann mapping as the mesh size of the hexagonal circle packing tends to zero. 
		
		In the smooth theory, holomorphic functions are related to many classical surfaces in differential geometry. For example, the Weierstrass representation asserts that every minimal surface in Euclidean space can be represented by a pair of holomorphic functions. Robert Bryant \cite{Bryant1987} obtained an analogous representation for surfaces in hyperbolic space with constant mean curvature $H\equiv 1$ (CMC-1 surfaces for short), where horospheres are trivial CMC-1 surfaces. 

       It gives rise to a question: can one obtain discretization of classical surfaces from discrete holomorphic data, e.g. from a pair of circle packings? It will be beneficial to constructing interesting examples of smooth surfaces by discrete approximation.
       
       A rich theory of discrete surfaces has been developed in terms of quadrilateral meshes via integrable systems. Many classical surfaces in space, like CMC surfaces, possess integrable system structures which provide recipes for construction. These construction often requires a particular parametrization of the surfaces, like isothermic coordinates. As a discretization, quadrilateral meshes are considered with edges playing the role of principal curvature directions. For example, Bobenko and Pinkall \cite{Bobenko1996} considered quadrilateral meshes in the plane such that each face has factorized cross ratios. Each of the meshes can be used to construct a polyhedral surface in $\mathbb{R}^3$ as a discrete minimal surface \cite{Bobenko2006} and in $\mathbb{H}^3$ as a discrete CMC-1 surface \cite{Hertrich-Jeromin2000}. As remarked in \cite{Bobenko2002}, the underlying construction relies on a solution to a discrete Toda-type equation \cite{Adler2001} (see Definition \ref{def:Toda}). However, most meshes in the plane do not admit such a solution and so their construction is not applicable generally. It remains a question to have an alternative recipe for general meshes in the plane.       
		
		In this article, we consider circle patterns. Given a planar graph, a \textbf{circle pattern} is a placement of the vertices in the plane such that each face has a circumcircle passing through the vertices. In this way, there is a corresponding circle for every face, a pair of intersecting circles for every edge and an intersecting point of circles for every vertex of the graph. It generalizes circle packings in the sense that every circle packing with triangular interstice together with its dual circle packing forms a circle pattern. In this case, tangent points of the primal circle packing are the vertices of the planar graph under our consideration. Every edge corresponds to a pair of a primal circle and a dual circle. Every face corresponds to either a primal circle or a dual circle. In particular, faces correspond to dual circles are triangular.
		
		 By triangulating the faces of a planar graph, circle patterns are parameterized by \textbf{complex cross ratios}. We denote by $M=(V, E, F)$ a triangulation of a disk with or without boundary where $V$, $E$ and $F$ are the sets of vertices, edges and faces respectively. Vertices are denoted by $i,j,k$. An unoriented edge is denoted by $\{ij\}=\{ji\}$ indicating that its end points are vertices $i$ and $ j$. Given a realization $z:V \to \mathbb{C}\cup\{\infty\}$ on the Riemann sphere (i.e. a placement of the vertices), we associate a complex cross ratio to every common edge $\{ij\}$ shared by triangles $\{ijk\}$ and $\{jil\}$:
		\[
		X_{ij} :=  -\frac{(z_k - z_i)(z_l -z_j)}{(z_i - z_l)(z_j - z_k)} = X_{ji}
		\]
		which encodes how the circumdisk of triangle $z_iz_jz_k$ is glued to that of $z_jz_iz_l$ (See Figure \ref{fig:orientation}). It defines a function  $X: E_{int}  \to \mathbb{C}$ on interior edges, which can be characterized as solutions to a system of polynomial equations (See \cite{Lam2019}). One can show that two circle patterns differ by a M\"{o}bius transformation if and only if their cross ratios are identical. Furthermore, a circle pattern is \textbf{Delaunay} if neighboring circles intersect with angles $\Arg X \in [0,\pi)$. 
		
		\begin{figure}
			\centering
			\begin{tikzpicture}[line cap=round,line join=round,>=triangle 45,x=1.0cm,y=1.0cm,scale=1.5]
				\clip(-1.5440494751930037,0.247159828344958221) rectangle (3.00498188597834266,3.7102277223684125);
				\draw [shift={(0.78,3.26)},line width=.5pt,fill=black,fill opacity=0.10000000149011612] (0,0) -- (-24.813257032322223:0.2512792886121933) arc (-24.813257032322223:22.61478568507276:0.2512792886121933) -- cycle;
				\draw [line width=.5pt] (0.78,3.26)-- (-1.14,1.18);
				\draw [line width=.5pt] (-1.14,1.18)-- (0.78,0.56);
				\draw [line width=.5pt] (0.78,0.56)-- (0.78,3.26);
				\draw [line width=.5pt] (0.78,3.26)-- (2.56,1.1);
				\draw [line width=.5pt] (2.56,1.1)-- (0.78,0.56);
				\draw [line width=.5pt,dash pattern=on 2pt off 2pt] (0.1558333333333333,1.91) circle (1.4873076439586324cm);
				\draw [line width=.5pt,dash pattern=on 2pt off 2pt] (1.3423595505617978,1.91) circle (1.462445986731841cm);
				\draw [shift={(1.3667829962387972,1.5629195801280533)},line width=.5pt]  plot[domain=-4.026659469478403:1.1232763516377273,variable=\t]({1.*0.27341791713268293*cos(\t r)+0.*0.27341791713268293*sin(\t r)},{0.*0.27341791713268293*cos(\t r)+1.*0.27341791713268293*sin(\t r)});
				\draw [shift={(0.06824381377497166,1.6013948151640185)},line width=.5pt]  plot[domain=-3.901355408465564:1.056345013735869,variable=\t]({1.*0.26534665582572997*cos(\t r)+0.*0.26534665582572997*sin(\t r)},{0.*0.26534665582572997*cos(\t r)+1.*0.26534665582572997*sin(\t r)});
				\draw [->,line width=.5pt] (1.4822087013466927,1.803389799102836) -- (1.3379265699618232,1.8611026516567828);
				\draw [->,line width=.5pt] (0.19328832764185855,1.8226274166208185) -- (0.029768578739006457,1.899577886692748);
				\draw (0.676232923642945,0.499158511750158) node[anchor=north west] {$i$};
				\draw (0.676232923642945,3.6193820462352503) node[anchor=north west] {$j$};
				\draw (-1.3979869928203247,1.2073567332806394) node[anchor=north west] {$k$};
				\draw (2.481113458984279,1.1319729466969817) node[anchor=north west] {$l$};
				\draw (1.0091662489955112,3.3518625204937885) node[anchor=north west] {$\Arg X_{ij}$};
				\begin{scriptsize}
					\draw [fill=black] (0.78,0.56) circle (1.pt);
					\draw [fill=black] (0.78,3.26) circle (1.pt);
					\draw [fill=black] (-1.14,1.18) circle (1.pt);
					\draw [fill=black] (2.56,1.1) circle (1.pt);
				\end{scriptsize}
			\end{tikzpicture}
			\caption{Two neighboring triangles sharing the edge $\{ij\}$ together with circumscribed circles.}
			\label{fig:orientation}
		\end{figure}

        Every pair of circle patterns induce an \textbf{osculating M\"{o}bius transformation}. Given two realizations $z$ and $\tilde{z}$ of a triangle mesh in the plane, there is a unique M\"{o}bius transformation $A_{ijk}$ for every face $\{ijk\}$ mapping vertices $z_i,z_j,z_k$ to $\tilde{z}_i,\tilde{z}_j,\tilde{z}_k$. It induces a map $A: F \to SL(2,\mathbb{C})/\{\pm I\}$ and we call it the osculating M\"{o}bius transformations from $z$ to  $\tilde{z}$ (See Section \ref{sec:discreteosc}). The map $A$ can be interpreted as being defined on the vertex set $V^*$ of the dual graph. Regarding $\mathbb{H}^3$ as a homogeneous space with projection $\mathfrak{i} :SL(2,\mathbb{C})\to SL(2,\mathbb{C})/SU(2) \cong \mathbb{H}^3$, we then obtain a realization of the dual graph $f:= \mathfrak{i} \circ A$ into hyperbolic space. Our goal is to investigate how such realizations are related to surfaces with constant mean curvature $H\equiv1$.
        
        The discrete CMC-1 surfaces under consideration are certain horospherical nets in hyperbolic space. A \textbf{horospherical net} is a realization of a graph in $\mathbb{H}^3$ such that the vertices of each face lies on the same horosphere. It is a piecewisely horospherical surface and differentiable everywhere except along edges. Thus the classical pointwise mean curvature is not defined over edges. However, by considering the parallel surfaces of a horospherical net, one can consider the integrated mean curvature over each face as the infinitesimal change of the face area analogous to the Steiner formula \cite{Bobenko2010a}. Such a quantity can be expressed in terms of the edge lengths $\ell$ and dihedral angles $\alpha$ between neighboring horospherical faces. 
        \begin{definition}\label{def:intmean}
        	Given a horospherical net $f:V^* \to \mathbb{H}^3$, we define its integrated mean curvature $H:F^* \to \mathbb{R}$ over faces
        	\[
        	H_\phi := \area(f(\phi)) + \frac{1}{2}  \sum_{ij \in \phi} \ell_{ij} \tan \frac{\alpha_{ij}}{2} \quad \forall \phi \in F^*
        	\]
        	When $f$ is trivalent with non-vanishing edge lengths, the integrated mean curvature is equivalent to 
        	\[
        	H_{\phi} =- \frac{1}{2} \frac{d}{dt} \area(f_t(\phi)) |_{t=0} 
        	\]
        	where each $f_t$ is the horospherical net parallel to $f$ at distance $t$. 
        \end{definition}
      If the ratio of the integrated mean curvature to the face area is constantly equal to one, we call the horospherical net a discrete CMC-1 surface (see Definition \ref{def:discretecmc}). Our main result is a Weierstrass-type representation of a discrete CMC-1 surface in terms of a pair of circle patterns (See Figure \ref{fig:Weierstrass}. For any matrix $A\in SL(2,\mathbb{C})$, we write $A^*$ the hermitian conjugate.
      \begin{figure}[h]
      	\centering
      	\begin{minipage}{.47\textwidth}
      		\includegraphics[width=0.99\textwidth]{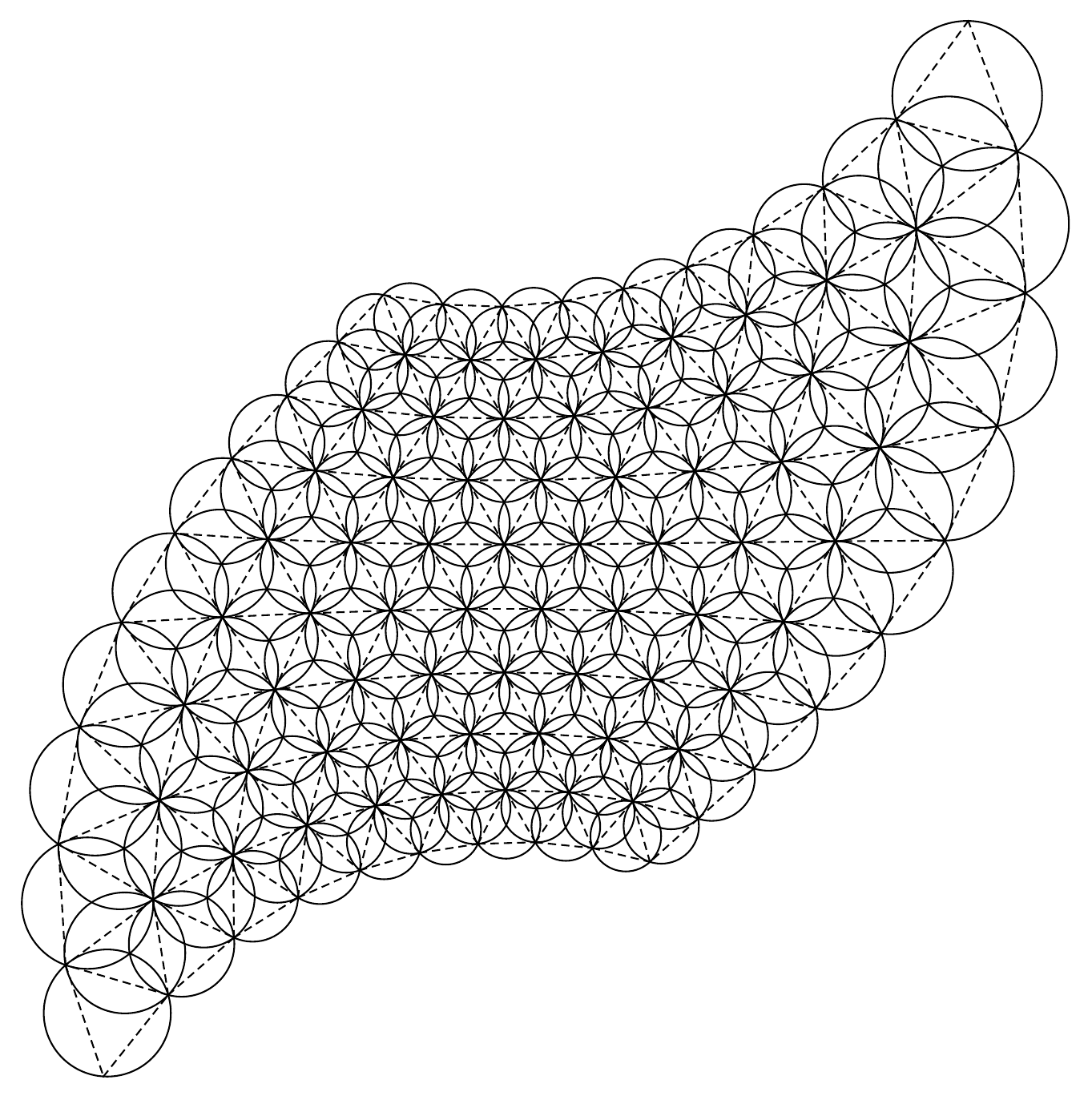}
      	\end{minipage}
      	\begin{minipage}{.47\textwidth}
      		\includegraphics[width=0.99\textwidth]{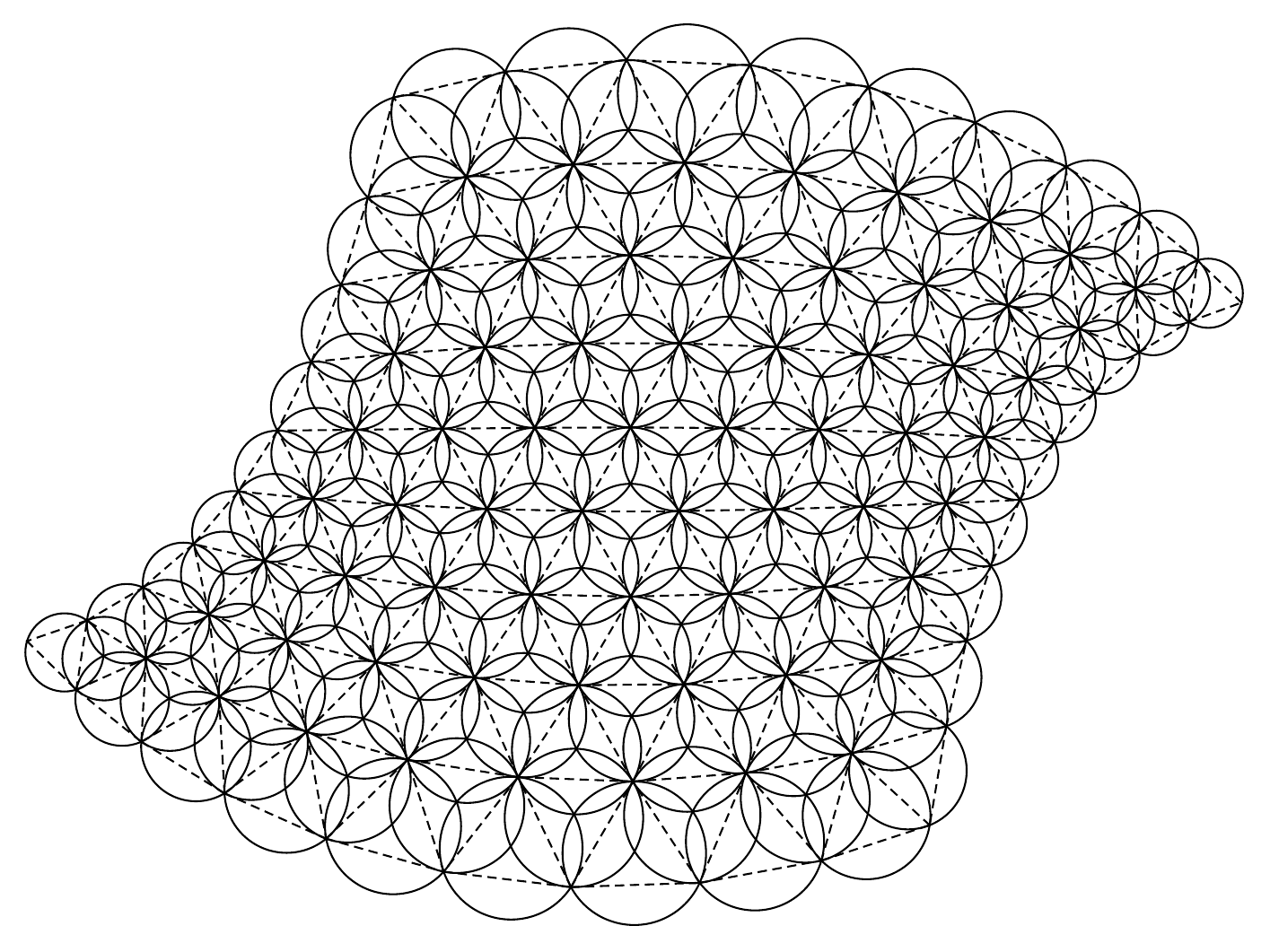}
      	\end{minipage}
      	\includegraphics[width=0.6\textwidth]{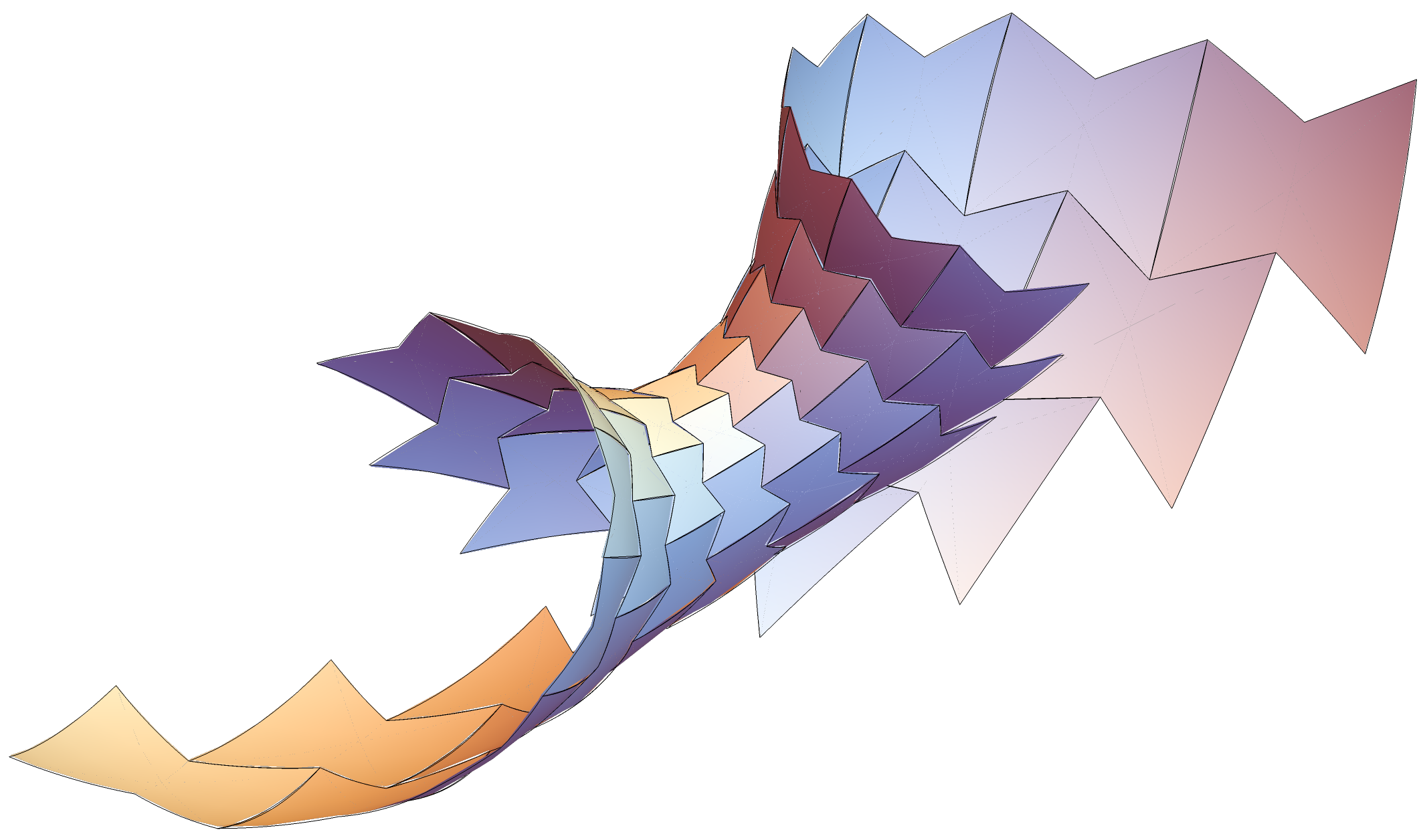}
      	\caption{The top row shows two circle patterns sharing the same modulus of cross ratios, where the triangle meshes are indicated by dotted lines. The osculating M\"{o}bius transformation induces a realization of the dual graph into hyperbolic space (bottom). It forms a surface consisting of pieces of horospheres. Over each face, the ratio of the integrated mean curvature (Definition \ref{def:intmean}) to the face area is constantly equal to $1$. The hyperbolic Gauss map is the vertices of the circle pattern on the top right. Such a correspondence is a discrete analogue of the Weierstrass representation for CMC-1 surfaces in hyperbolic space.}
      	\label{fig:Weierstrass}
      \end{figure}
        
      \begin{theorem}\label{thm:horo}
      	We denote a triangulation by $(V,E,F)$ and its dual mesh by $(V^*,E^*,F^*)$. Given two Delaunay circle patterns $z,\tilde{z}:V\to \mathbb{C}$ with cross ratios $X,\tilde{X}$ such that $|X| \equiv |\tilde{X}|$. Let $A:V^{*} \to SL(2,\mathbb{C})$ be the osculating M\"{o}bius transformation from $z$ to $\tilde{z}$. Then the realization $f:V^{*} \to \mathbb{H}^3$ of the dual graph 
      	defined by
      	\[
      	f:= A A^{*}
      	\]
      	is a horospherical net with integrated mean curvature over faces satisfying
      	\[
      	\frac{H_{\phi}}{\area(f(\phi))} \equiv 1  \quad  \forall \phi \in F^*.
      	\]
      	In particular $f$ is a discrete CMC-1 surface with hyperbolic Gauss map $\tilde{z}$.
      	
      	Conversely, suppose $f:V^{*} \to \mathbb{H}^3$ is discrete CMC-1 surface with hyperbolic Gauss map $\tilde{z}:V \to \mathbb{C}$. Then there exists a Delaunay circle pattern $z$ such that 
      	\[
      	f=A A^{*}
      	\] where $A$ is the osculating M\"{o}bius transformation from $z$ to $\tilde{z}$. The cross ratios $X,\tilde{X}$ of the circle patterns satisfy $|X| = |\tilde{X}|$.	
      \end{theorem}
  
 Two circle patterns satisfying $|X| = |\tilde{X}|$ are also said to differ by vertex scaling (see \cite{Luo2004}).
      
        Our construction generalizes the integrable system approach in \cite{Hertrich-Jeromin2000}. There, the construction relies on solutions to the discrete Toda-type equation (See Definition \ref{def:Toda}). We shall illustrate how the solutions related to our construction. Indeed for every such solution, we show that it induces a 1-parameter family of circle patterns, which contains many pairs of circle patterns with cross ratios satisfying $|X|\equiv|\tilde{X}|$ (Theorem \ref{thm:inttoda}). The resulting discrete CMC-1 surfaces include the ones considered by Hertrich-Jeromin \cite{Hertrich-Jeromin2000} via the integrable system approach. In particular, our construction is applicable even for meshes in the plane that does not support a solution to the discrete Toda-type equation.
        
        We then establish convergence of discrete CMC-1 surfaces. By restricting the combinatorics to triangle lattices, we show that every smooth CMC-1 surfaces without umbilic points can be approximated by our discrete CMC-1 surfaces (Theorem \ref{thm:converge}). It relies on the results by He-Schramm \cite{Schramm1998} and B\"{u}cking \cite{Bucking2016}.

       Instead of sharing the same modulus of cross ratios $|X|=|\tilde{X}|$, we also investigate an alternative theory in terms of a pair of circle patterns sharing the same intersection angles $\Arg X = \Arg \tilde{X}$ (Theorem \ref{thm:equid}). We analogously obtain a Weierstrass-type representation, though the resulting realizations in hyperbolic space are no longer horospherical nets.
       
         In contrast to the smooth theory, the Weierstrass data of discrete CMC-1 surfaces in hyperbolic space is different from that of discrete minimal surfaces in $\mathbb{R}^3$. In \cite{Lam2017,Lam2018}, it is shown that every discrete minimal surface corresponds to an infinitesimal deformation of a circle pattern. For a discrete CMC-1 surface, its Weierstrass data consists of a pair of circle patterns.        
         
            Throughout the article, we focus on the local theory of discrete CMC-1 surfaces, where surfaces are assumed to be simply connected. For surfaces with non-trivial topology, the construction of discrete CMC-1 surfaces involves the period problem and it requires to study the deformation space of circle patterns. Such a problem is related to Kojima-Tan-Misuhima's conjecture \cite{Kojima2003, Lam2019}.
            
        
       Osculating M\"{o}bius transformations between circle patterns play an important role in the construction of discrete CMC-1 surfaces. In the smooth theory, osculating M\"{o}bius transformations is essential to study complex projective structures \cite{Anderson1998} and the renormalized volume of hyperbolic 3-manifolds \cite{Brock2019}. It involves the Epstein map relating surfaces in $\mathbb{H}^3$ and the conformal metrics at infinity $\partial \mathbb{H}^3 \cong S^2$. We show that our approach to discrete CMC-1 surfaces is a special case of such a map (See Section \ref{sec:horosphersver}).
%
        
        The outline of the paper is as follows. In section \ref{sec:background}, we review the connection between classical osculating M\"{o}bius transformations and smooth CMC-1 surfaces. In section \ref{sec:discreteosc}, we develop properties of osculating M\"{o}bius transformations between circle patterns. In section \ref{sec:discretecmc}, horospherical nets and integrated mean curvature are introduced. We then prove the main result about the correspondence between circle patterns and discrete CMC-1 surfaces. In section \ref{sec:toda}, we explain how our discrete CMC-1 surfaces generalize the previous work in the integrable system approach. In section \ref{sec:convergence}, we prove convergence of discrete CMC-1 surfaces. In section \ref{sec:circarg}, an alternative theory for circle patterns sharing the same intersection angles is established. In section \ref{sec:minimal}, we explain the connection between minimal surfaces in $\mathbb{R}^3$ and osculating M\"{o}bius vector fields.

\section{Background} \label{sec:background}

We review some notations and results of smooth osculating M\"{o}bius transformations related to CMC-1 surfaces.

\subsection{M\"{o}bius transformations} \label{sec:mobius}

Without further notice, we consider orientation-preserving M\"{o}bius transformations only. They are in the form $z \mapsto (a z + b)/(c z +d)$ for some $a,b,c,d \in \mathbb{C}$ such that $ad - bc =1$ and are also called complex projective transformations. Every M\"{o}bius transformation represents an element in $SL(2,\mathbb{C})/\{\pm I\}$. These transformations are generated by Euclidean motions and inversion $z \mapsto 1/z$. They are holomorphic, map circles to circles and preserve cross ratios. 	

We also make use of the fact there exists a unique M\"{o}bius transformation $A \in SL(2, \mathbb{C})/\{\pm \emph{I}\}$ that maps any three distinct points $z_i$ in the plane to any other three distinct points $\tilde{z}_i$ determined via
\[
A \left(\begin{array}{c}
z_i \\ 1
\end{array}\right) = \lambda_{i} \left(\begin{array}{c}
\tilde{z}_i \\ 1
\end{array}\right) \quad \text{for some } \lambda_i \in \mathbb{C}-\{0\}.
\]

\subsection{Cross ratios}

As mentioned in the introduction, every triangle mesh $z:V \to \mathbb{C}$ is associated with cross ratios $X: E_{int} \to \mathbb{C}\cup \{\infty\}$ defined over interior edges. Recall that for ever common edge $\{ij\}$ shared by triangles $\{ijk\}$ and $\{jil\}$
	\[
X_{ij} :=  -\frac{(z_k - z_i)(z_l -z_j)}{(z_i - z_l)(z_j - z_k)} = X_{ji}
\]
See Figure \ref{fig:orientation}.

\begin{proposition}[\cite{Lam2019}]
	Suppose $M=(V,E,F)$ is a triangulation of a surface with or without boundary. Then a function $X: E_{int} \to \mathbb{C}$ is the cross ratio of some triangle mesh $z:V \to \mathbb{C}\cup \{\infty\}$  if and only if for every interior vertex $i$ with adjacent vertices numbered as $1$, $2$, ..., $n$ in the clockwise order counted from the link of $i$,
	\begin{gather}
	\Pi_{j=1}^n X_{ij} =1  \label{eq:crproduct}\\
	X_{i1} + X_{i1} X_{i2} + X_{i1}X_{i2}X_{i3} + \dots +  X_{i1}X_{i2}\dots X_{in} =0 \label{eq:crsum}
	\end{gather}
	In particular, the realization $z$ is unique up to a M\"{o}bius transformation.
\end{proposition}
	Particularly, the triangle mesh has no branching at vertices if for every vertex $i$
\[
\sum_j \Arg X_{ij} = 2 \pi.
\]
Throughout the article, we only consider triangle meshes without branching vertices for simplicity. Indeed, the cross ratio provides a recipe to glue neighboring circumdisks associated to faces. Equations \eqref{eq:crproduct} and \eqref{eq:crsum} ensures that the holonomy around each vertex under the gluing construction is trivial.
%

%
%

\subsection{Hyperbolic space}\label{sec:hyperbolic}
We start with the hyperboloid model. We denote by $\mathbb{R}^{3,1}$ the Minkowski space, which is a four dimensional real vector space equipped with the Minkowski inner product
\[
\langle (x_0,x_1,x_2,x_3),(y_0,y_1,y_2,y_3) \rangle := - x_0 y_0 + x_1 y_1  + x_2 y_2 + x_3 y_3
\] 
and the hyperbolic 3-space is the subset
\[
\mathbb{H}^3 = \{(x_0 ,x_1,x_2,x_3) \in \mathbb{R}^{3,1}|  - x_0^2 + x_1^2  + x_2^2 + x_3^2 =-1, x_0 >0 \}
\]
together with the induced metric. The Minkowski space is isomorphic to the real vector space of 2 by 2 Hermitian matrices denoted as $\Herm(2)$ via
\[
(x_0,x_1,x_2,x_3) \leftrightarrow U=\left( \begin{array}{cc}
x_0 +x_3 & x_1 + \mathbf{i} x_2 \\ x_1 - \mathbf{i} x_2 & x_0 -x_3
\end{array}\right)
\]
equipped with the bilinear form
\[
\langle U,V \rangle = -\frac{1}{2} \trace(U \tilde{V})
\]
where $\mathbf{i}=\sqrt{-1}$ and $\tilde{V}$ is the cofactor matrix of $V$ defined by $\tilde{V} V = \det(V) \emph{I}$. In particular, $||U||^2= \langle U, U \rangle = -\det(U)$.

The special linear group $SL(2,\mathbb{C})$ acts on $\Herm(2) \cong \mathbb{R}^{3,1}$. For any $A \in SL(2,\mathbb{C})$, we denote  by $A^*$ the Hermitian conjugate. The mapping
\[
V \in \Herm(2) \mapsto A V A^*
\]
preserves the Minkowski inner product. Particularly, it acts on $\mathbb{H}^3$ isometrically and transitively. 

 The hyperbolic space is identified as the space of hermitian matrices with determinant $1$ and positive trace, which thus can be obtained from $SL(2,C)$ via a mapping
\begin{align} \label{eq:SLhyp}
A \mapsto A A^*  \in \Herm(2).
\end{align}
Two elements $A$ and $\tilde{A}$ induce the same hermitian matrix if and only if $\tilde{A} = A B$ for some $B \in SU(2)$. It yields an identification of the hyperbolic space as the left cosets of $SU(2)$ in $SL(2,\mathbb{C})$
\[
\mathfrak{i}: SL(2,\mathbb{C}) \to SL(2,\mathbb{C})  / SU(2) \cong \mathbb{H}^3 .
\]

We denote the upper light cone by
\[
L^{+}:= \{ (x_0 ,x_1,x_2,x_3) \in \mathbb{R}^{3,1}|  - x_0^2 + x_1^2  + x_2^2 + x_3^2 =0, x_0 >0 \}
\]
which corresponds to the set of hermitian matrices $U$ satisfying $\det U =0$ and $\trace U >0$. Every element in the upper light cone defines a horosphere $H_U$ in hyperbolic space
\[
U \in L^{+} \leftrightarrow H_U:=\{ x \in \mathbb{H}^3 | \langle x,U \rangle =1\}.
\]

\subsection{Totally umbilical hypersurfaces} \label{sec:umbilic} A hypersurface in hyperbolic space is totally umbilic if at every point the normal curvatures in all directions are the same. It turns out that there are four types of complete totally umbilical hypersurfaces which are characterized by their constant mean curvature $H$. 
\begin{itemize}
	\item Geodesic sphere ($H>1$)
	\item Horosphere ($H=1$)
	\item Equidistant ($1>H>0$)
	\item Totally geodesic hyperplane ($H=0$)
\end{itemize}

These hypersurfaces are easily visualized in the Poincaré ball model: A geodesic sphere is a Euclidean sphere that is disjoint from $\partial \mathbb{H}^3$. A horosphere is a sphere that touches $\partial \mathbb{H}^3$. A totally geodesic hyperplane is a spherical cap that intersects $\partial \mathbb{H}^3$ orthogonally. An equidistant is a sphere that intersect $\partial \mathbb{H}^3$ but not orthogonally. Every equidistant has a constant distance from the totally geodesic hyperplane that share the sames intersection on $\partial \mathbb{H}^3$. Throughout the following sections, horospheres and equidistants are frequently used.

\subsection{Smooth osculating M\"{o}bius transformations}\label{sec:smoothM}

It is a classical result \cite{Thurston1986,Small1994,Anderson1998} that a locally univalent holomorphic function is associated with an osculating M\"{o}bius transformation. We review its definition and property.

%

We assume $\Omega$ is a simply connected domain in $\mathbb{C}$. A holomorphic function $h: \Omega \to \mathbb{C}$ is \emph{locally univalent} if $h'$ non-vanishing. It is associated with the \textit{osculating M\"{o}bius transformation} $A_h:\Omega \to SL(2,\mathbb{C})$, which is a continuous mapping satisfying for each $z \in \Omega$,
\begin{align*}
A_{h}(z) = \left(\begin{array}{cc}
\alpha(z) & \beta(z) \\ \gamma(z) & \delta(z)
\end{array} \right)
\end{align*}
induces the unique M\"{o}bius transformation that coincides with the 2-jet of $h$ at $z$, i.e.
\begin{align*}
h(z) =& \frac{\alpha z + \beta}{\gamma z + \delta} \\
h'(z) =& \frac{\partial}{\partial w}(\frac{\alpha w + \beta}{\gamma w + \delta})|_{w=z} \\
h''(z) =& \frac{\partial^2}{\partial w^2}(\frac{\alpha w + \beta}{\gamma w + \delta})|_{w=z} 
\end{align*}
Together with $\alpha \delta - \beta \gamma =1$, it yields
\begin{align} \label{eq:smoothmob}
A_{h}= \frac{1}{ h'(z)^{3/2}} \left(
\begin{array}{cc}
h'(z)^2-\frac{h(z) h''(z)}{2} & \frac{z h(z) h''(z)}{2 }+ h(z)
h'(z) - z h'(z)^2\\
-\frac{h''(z)}{2 } & \frac{z h''(z)}{2 } + h'(z) \\
\end{array}
\right).
\end{align}
Notice that there is a square root in the denominator. Pointwisely, one can pick one of the two branches arbitrarily. In order to define a continuous map to $SL(2,\mathbb{C})$, it is natural to pick the branches consistently. Generally, if  $\Omega$ is a multiply connected region, the map $A_h$ is only defined on a double cover of $\Omega$ due to the two branches of the square root. As we shall see in the next section, a similar phenomenon occurs in the discrete case. 

We denote the Schwarzian derivative of $h$ by
\[
S_h(z):= \frac{h'''(z)}{h'(z)} - \frac{3}{2} \left( \frac{h''(z)}{h'(z)} \right)^2.
\]
By direct computation, the Maurer-Cartan form of $A_h$ is the differential form
\begin{equation}\label{eq:mauerh}
	A_h^{-1} \, dA_h = -\frac{S_h(z)}{2}\left( \begin{array}{cc}
	z & -z^2 \\ 1 & -z
	\end{array}\right) dz.
\end{equation}
It is a $sl(2,\mathbb{C})$-valued 1-form on $\Omega$ whose determinant as a quadratic form vanishes. It is equivalent to say that the pull-back of the Killing form on $sl(2,\mathbb{C})$ via $A_h$ vanishes. The 1-form measures how $A_h$ differs from being constant and hence how $h$ differs from a M\"{o}bius transformation.

Given locally univalent functions $g:\Omega \to \mathbb{C}$ and $h:g(\Omega) \to \mathbb{C}$, we have the composition $h\circ g:\Omega \to \mathbb{C}$. Notice that $A_h$ is defined on $g(\Omega)$. It is known that osculating M\"{o}bius transformations satisfy a composition rule \cite{Anderson1998}
\begin{equation}\label{eq:composition}
	A_{h\circ g} = (A_h \circ g) A_g
\end{equation}
where the right side involves matrix multiplication.

For the purpose of this article, we are interested in the osculating M\"{o}bius transformation between a pair of locally univalent functions $g, \tilde{g}: \Omega \to \mathbb{C}$. For every $w \in \Omega$, there is a  neighborhood $U$ such that $g|_U, \tilde{g}|_U$ are injective. Denoting the composition $h:=(\tilde{g}|_U) \circ (g|_U)^{-1}$, we have the osculating M\"{o}bius transformation $A_h$ defined on $g(U)$. In this way, we define the osculating M\"{o}bius transformation $A:\Omega \to SL(2,C)$ from $g$ to $\tilde{g}$ as the composition
\[
A|_U := A_h \circ g|_U = (A_{\tilde{g}} \, A^{-1}_{g})|_U
\]
where Formula \eqref{eq:composition} is used.

\begin{proposition}\label{prop:smoothoscu}
	Suppose $g,\tilde{g}:\Omega \to \mathbb{C}$ are locally univalent functions on a simply connected region. The osculating M\"{o}bius transformation $A:\Omega \to SL(2,\mathbb{C})$ from $g$ to $\tilde{g}$ is given by
	\[
	A=A_{\tilde{g}} \, A^{-1}_{g}
	\]
Its Mauer-Cartan form is
\[
A^{-1} dA = \frac{S_g(w)-S_{\tilde{g}}(w)}{2 g'(w)}\left( \begin{array}{cc}
g(w) & -g(w)^2 \\ 1 & -g(w)
\end{array}\right) dw
\]
which is the pull-back of $A_h^{-1} dA_h$ via $g$ by setting $z=g(w)$ in Equation \eqref{eq:mauerh}.
\end{proposition}

Osculating M\"{o}bius transformations between a pair of locally univalent functions are closely related to CMC-1 surfaces in hyperbolic space.

\subsection{Smooth CMC-1 surfaces in hyperbolic space}\label{sec:smoothCMC}

By considering moving frames, Robert Bryant deduced a Weierstrass representation of CMC-1 surfaces in terms of holomorphic data \cite{Bryant1987}. The goal of this section is to interpret the formula in terms of osculating M\"{o}bius transformations. It is closely related to holomorphic null curves in $PSL(2,\mathbb{C})$ (See \cite{Small1994}).

Recall that for a smooth surface $f:\Omega \to \mathbb{H}^3$ with unit normal vector field $N$, the \emph{hyperbolic Gauss map} sends each $p\in \Omega$ to a point in $G(p)\in \partial \mathbb{H}^3\cong S^2 $ along the oriented normal geodesics that starts at $f(p)$ in the direction of $N(p)$. For a fixed horosphere, we orient its unit normal so that its hyperbolic Gauss map is constant and the image is the tangency point of the horosphere with $\partial \mathbb{H}^3$.

In the following, the hyperbolic space is identified as a subset of the space of 2 by 2 Hermitian matrices (see Section \ref{sec:hyperbolic}).

\begin{proposition}\cite{Bryant1987,Umehara1993} \label{thm:bryant}
	Suppose $\Omega \subset \mathbb{C}$ is simply connected and $f:\Omega \to \mathbb{H}^3$ is a conformal immersion of a CMC-1 surface. Then there exists a holomorphic immersion $A:\Omega \to SL(2,\mathbb{C})$ such that $f = A A^*$  and 
	\begin{equation}\label{eq:bryant}
		A^{-1} dA = \left( \begin{array}{cc}
	g(w) & -g(w)^2 \\ 1 & -g(w)
	\end{array}\right)  \eta
	\end{equation}
	for some meromorphic function $g$ and a holomorphic 1-form $\eta$ on $\Omega$. The Hopf differential $Q$ satisfies
	\[
	Q = \eta\,  dg = \eta\, g'(w) \,dw
	\]
	and the hyperbolic Gauss map is
		\[
	G(w) = \frac{\frac{\partial (A)_{11}}{\partial w}}{\frac{\partial (A)_{21}}{\partial w}}
	\]
	where $A_{ij}$ is the $(i,j)$-th entry of the matrix $A$.
	
	Conversely, every meromorphic function $g$ and a holomorphic 1-form $\eta$ determine a CMC-1 surface in $\mathbb{H}^3$ via \eqref{eq:bryant}.
\end{proposition}

The \emph{Hopf differential} $Q$ is a holomorphic quadratic differential. It describes the trace-free part of the second fundamental form of $f$ and vanishes at umbilic points, where the normal curvatures in all directions are equal. The 1-form $\eta$ is nowhere vanishing since $A$ is an immersion. Hence $f(p)$ is an umbilic point if and only if $g'(p)=0$. 

We rephrase Proposition \ref{thm:bryant} into a form that is comparable to its discrete counterpart  (Theorem \ref{thm:horo}).

\begin{proposition}\label{pro:osccmc}
	Suppose $f:\Omega \to \mathbb{H}^3$ is a conformal immersion of an umbilic-free CMC-1 surface. Then there exists a pair of locally univalent functions $g,\tilde{g}: \Omega \to \mathbb{C}\cup\{\infty\}$ such that
	\begin{equation}\label{eq:productA}
			f  = A  A^*
	\end{equation}
	where $A: \Omega \to SL(2,\mathbb{C})$ is the osculating M\"{o}bius transformation from $g$ to $\tilde{g}$ satisfying
	\begin{equation}\label{eq:osccomp}
		A = A_{\tilde{g}} A_{g}^{-1}
	\end{equation}
	Furthermore, we have the Hopf differential $
	Q = (S_g - S_{\tilde{g}})/2$ and the hyperbolic Gauss map $G = \tilde{g}$.
	
	Conversely, every pair of locally univalent functions induce a conformal immersion of an umbilic-free CMC-1 surface via Equation \eqref{eq:productA} and \eqref{eq:osccomp}.
\end{proposition}
\begin{proof}
Consider the functions $A, g, \eta$ as in Theorem \ref{thm:bryant}. 
Notice that $g'$ is non-vanishing since $f$ is umbilic free. By considering the Hill's equation, it is known that there exists a locally univalent function $g^{\dagger}$ that solves
\[
\eta =  \frac{S_g(w)-S_{g^{\dagger}}(w)}{2 g'(w)} dw.
\]
The map $g^{\dagger}$ is unique up to a post-composition with a M\"{o}bius transformation. Writing $\tilde{A}:= A_{g^{\dagger}} A^{-1}_g$. We have
\[
A^{-1} dA = \tilde{A}^{-1} d\tilde{A}
\] 
and hence $A= C \tilde{A}$ for some constant $C \in SL(2,C)$. We define $\tilde{g}$ the post-composition of $g^{\dagger}$ with the M\"{o}bius transformation $C$ and obtain
\[
A = A_{\tilde{g}} A^{-1}_g
\]
Hence $A$ is the osculating M\"{o}bius transformation from $g$ to $\tilde{g}$.

The mapping $A$ can be written explicitly in terms of $g$ and $\tilde{g}$. By direct computation, we have the hyperbolic Gauss map
\[
G = \frac{\frac{\partial (A)_{11}}{\partial w}}{\frac{\partial (A)_{21}}{\partial w}}=\tilde{g}
\]
\end{proof}

\section{Discrete osculating M\"{o}bius transformations} \label{sec:discreteosc}

In this section, we define an osculating M\"{o}bius transformation between two Delaunay circle patterns, which induces a realization of the dual graph in $SL(2,\mathbb{C})/\{\pm I\}$. In the subsections, we further investigate lifting to  $SL(2,\mathbb{C})$ and the image of horospheres under an osculating M\"{o}bius transformation.
	
	Analogues of osculating M\"{o}bius transformations between circle patterns have appeared in previous literature. He-Schramm \cite{Schramm1998} introduced ``contact transformations" in order to prove $C^{\infty}$-convergence of circle packings while Bobenko-Pinkall-Springborn \cite{Bobenko2010} considered ``discrete conformal maps".
%

\begin{definition}
	Suppose $z, \tilde{z}: V \to \mathbb{C}$ are two circle patterns with the same combinatorics. The osculating M\"{o}bius transformation from $z$ to $\tilde{z}$ is the mapping  $\tilde{A}: F \to SL(2,\mathbb{C})/\{\pm I\}$ such that for each face $\{ijk\}$, the M\"{o}bius transformation $\tilde{A}_{ijk}$ sends  $z_i,z_j,z_k$ to $\tilde{z}_i,\tilde{z}_j,\tilde{z}_k$ respectively. Explicitly, 
	\begin{align*}
	\tilde{A}_{ijk} = \frac{1}{a_{ijk}}
	\left(
	\begin{array}{cc}
	\frac{\tilde{z}_i \tilde{z}_j (z_j-z_i)+\tilde{z}_i \tilde{z}_k
		(z_i-z_k)+\tilde{z}_j \tilde{z}_k
		(z_k-z_j)}{(z_i-z_j) (z_j-z_k)
		(z_k-z_i)} & \frac{\tilde{z}_i \tilde{z}_j z_k
		(z_i-z_j)+\tilde{z}_i \tilde{z}_k z_j (z_k-z_i)+\tilde{z}_j
		\tilde{z}_k z_i (z_j-z_k)}{(z_i-z_j) (z_j-z_k)
		(z_k-z_i)} \\
	\frac{\tilde{z}_i (z_j-z_k)+\tilde{z}_j (z_k-z_i)+\tilde{z}_k
		(z_i-z_j)}{(z_i-z_j) (z_j-z_k)
		(z_k-z_i)} & \frac{\tilde{z}_i z_i (z_k-z_j)+\tilde{z}_j
		z_j (z_i-z_k)+\tilde{z}_k z_k
		(z_j-z_i)}{(z_i-z_j) (z_j-z_k)
		(z_k-z_i)}
	\end{array}
	\right)
	\end{align*}
	where 
	\[
	a_{ijk} = \pm \sqrt{ \frac{(\tilde{z}_i-\tilde{z}_j) (\tilde{z}_j-\tilde{z}_k)
			(\tilde{z}_k-\tilde{z}_i)}{ (z_i-z_j) (z_j-z_k)
			(z_k-z_i)}}.
	\]
\end{definition}

We consider an analogue of the Maurer-Cartan form, which also characterizes whether a mapping $\tilde{A}: F \to SL(2,\mathbb{C})/\{\pm I\}$ is an osculating M\"{o}bius transformation between circle patterns.

\begin{proposition}\label{prop:transit}
	Suppose $z, \tilde{z}: V \to \mathbb{C}$ are two circle patterns with cross ratios $X, \tilde{X}$. We denote by $\tilde{A}:F \to SL(2,\mathbb{C})/\{\pm I\}$ the osculating M\"{o}bius transformation from $z$ to $\tilde{z}$. Then for every edge oriented from vertex $i$ to $j$, with $\{ijk\}$ and $\{jil\}$ the left triangle and the right triangles, the transition matrix $\tilde{A}_{jil}^{-1} \tilde{A}_{ijk}$ has eigenvectors $(z_i,1)^T$ and $(z_j,1)^T$ with eigenvalues $\pm \lambda_{ij}$ and $\pm \lambda_{ij}^{-1}$ satisfying $\lambda_{ij}^2 = \frac{X_{ij}}{\tilde{X}_{ij}}=\lambda_{ji}^{2}$. Explicitly,
	\begin{equation*} 
	\tilde{A}_{jil}^{-1} \tilde{A}_{ijk}=\frac{1}{z_j-z_i}\left(
	\begin{array}{cc}
	\frac{z_j}{\lambda_{ij}} - \lambda_{ij} z_i & -z_i z_j (\frac{1}{\lambda_{ij}}-\lambda_{ij}) \\
	\frac{1}{\lambda_{ij}} - \lambda_{ij} & \lambda_{ij} z_j - \frac{z_i}{\lambda_{ij}}
	\end{array}
	\right)
	\end{equation*}
	
	Conversely, given a realization $\tilde{A}:F \to SL(2,\mathbb{C})/\{\pm I\}$, it is an osculating M\"{o}bius transformation between two circle patterns if there exists a mapping $z:V \to \mathbb{C}\cup \{\infty\}$ such that $z_i$ is a fixed point of the M\"{o}bius transformation $\tilde{A}_{jil}^{-1} \tilde{A}_{ijk}$ for every edge $\{ij\}$. The other circle pattern $\tilde{z}$ is determined such that $\tilde{z}_i$ is the image of $z_i$ under the M\"{o}bius transformation $\tilde{A}_{ijk}$ for any face $\{ijk\}$ containing vertex $i$.  
\end{proposition}

\begin{proof}

	We pick an arbitrary lift $A:F\to SL(2,\mathbb{C})$. Considering two neighboring faces $\{ijk\}$ and $\{jil\}$, both the M\"{o}bius transformations $\tilde{A}_{ijk},\tilde{A}_{jil}$ map $z_i$ to $\tilde{z}_i$ and maps $z_j$ to $\tilde{z}_j$. We deduce that $z_i,z_j$ are fixed points of the M\"{o}bius transformation $\tilde{A}_{jil}^{-1}\tilde{A}_{ijk}$.
	Thus $A^{-1}_{jil} A_{ijk}$ has eigenvectors $(z_i,1)^T$ and $(z_j,1)^T$ with eigenvalues $\lambda_{ij}$ and $\lambda_{ij}^{-1}$. We have
	\begin{align*}
		\tilde{X}_{ij} 		&=	-\frac{\det \left( \left(\begin{array}{c}
			\tilde{z}_k \\ 1
			\end{array}\right),   \left(\begin{array}{c}
			\tilde{z}_i \\ 1
			\end{array}\right)   \right)  }{\det \left( \left(\begin{array}{c}
			\tilde{z}_i \\ 1
			\end{array}\right), \left(\begin{array}{c}
			\tilde{z}_l \\ 1
			\end{array}\right)   \right)  }  \frac{\det \left( \left(\begin{array}{c}
			\tilde{z}_l \\ 1
			\end{array}\right), \left(\begin{array}{c}
			\tilde{z}_j \\ 1
			\end{array}\right)   \right)  }{\det \left( \left(\begin{array}{c}
			\tilde{z}_j \\ 1
			\end{array}\right), \left(\begin{array}{c}
			\tilde{z}_k \\ 1
			\end{array}\right)   \right)  } \\		
		 &=-\frac{\det \left( A_{ijk}\left(\begin{array}{c}
			z_k \\ 1
			\end{array}\right), A_{ijk}\left(\begin{array}{c}
			z_i \\ 1
			\end{array}\right)   \right)  }{\det \left(A_{ijk} \left(\begin{array}{c}
			z_i \\ 1
			\end{array}\right), A_{jil}\left(\begin{array}{c}
			z_l \\ 1
			\end{array}\right)   \right)  }  \frac{\det \left( A_{jil} \left(\begin{array}{c}
			z_l \\ 1
			\end{array}\right), A_{ijk}\left(\begin{array}{c}
			z_j \\ 1
			\end{array}\right)   \right)  }{\det \left(A_{ijk} \left(\begin{array}{c}
			z_j \\ 1
			\end{array}\right), A_{ijk} \left(\begin{array}{c}
			z_k \\ 1
			\end{array}\right)   \right)  } \\
		&=-\frac{\det \left( \left(\begin{array}{c}
			z_k \\ 1
			\end{array}\right), \left(\begin{array}{c}
			z_i \\ 1
			\end{array}\right)   \right)  }{\det \left(A_{jil}^{-1} A_{ijk} \left(\begin{array}{c}
			z_i \\ 1
			\end{array}\right), \left(\begin{array}{c}
			z_l \\ 1
			\end{array}\right)   \right)  }  \frac{\det \left( \left(\begin{array}{c}
			z_l \\ 1
			\end{array}\right), A_{jil}^{-1} A_{ijk}\left(\begin{array}{c}
			z_j \\ 1
			\end{array}\right)   \right)  }{\det \left( \left(\begin{array}{c}
			z_j \\ 1
			\end{array}\right),  \left(\begin{array}{c}
			z_k \\ 1
			\end{array}\right)   \right)  } \\
		&= X_{ij}/  \lambda_{ij}^2
	\end{align*}

The converse can be verified directly.
\end{proof}

Given an osculating M\"{o}bius transformation $\tilde{A}$ from one circle pattern $z$ to another $\tilde{z}$, the inverse $\tilde{A}^{-1}$ is then an osculating M\"{o}bius transformation from $\tilde{z}$ to $z$. Its transition matrix is expressed in terms of eigenvalues $\lambda$ and $\tilde{z}$. We shall use this in Section \ref{sec:discretecmc}.

\begin{corollary}
	If $A:F \to SL(2,\mathbb{C})$ is an osculating M\"{o}bius transformation from $z$ to $\tilde{z}$, then  the inverse $A^{-1}$ is the osculating M\"{o}bius transformation from $\tilde{z}$ to $z$. Its transition matrix is in the form
	\begin{equation} \label{eq:transit}
	(A^{-1}_{jil})^{-1} (A^{-1}_{ijk}) = A_{jil} A^{-1}_{ijk}= \frac{1}{\tilde{z}_j-\tilde{z}_i}\left(
	\begin{array}{cc}
	\lambda_{ij} \tilde{z}_j  - \frac{\tilde{z}_i}{\lambda_{ij} } & -\tilde{z}_i \tilde{z}_j (\lambda_{ij}- \frac{1}{\lambda_{ij}}) \\ \lambda_{ij} -\frac{1}{\lambda_{ij}}  & \frac{ \tilde{z}_j }{\lambda_{ij}}-  \lambda_{ij} \tilde{z}_i 
	\end{array}
	\right).
	\end{equation} 
	where $\lambda = \pm \sqrt{X/\tilde{X}}$.
\end{corollary}

As an analogue of the smooth theory (Equation \eqref{eq:composition}), discrete osculating M\"{o}bius transformations satisfy an obvious composition rule.
\begin{corollary}
	Suppose $z,\tilde{z},z^{\dagger}:V \to \mathbb{C}$ are three circle patterns with the same combinatorics. Then the osculating M\"{o}bius transformation from $z$ to $z^{\dagger}$ is the product of the osculating M\"{o}bius transformation from $z$ to $\tilde{z}$ with that from $\tilde{z}$ to $z^{\dagger}$.
\end{corollary}

\subsection{Coherent lifting to $SL(2,\mathbb{C})$}
We are interested in a lift $A: F \to SL(2,\mathbb{C})$ where signs over faces are chosen consistently so that the map $A$ is as ``continuous" as in the smooth theory (Equation \eqref{eq:smoothmob}). Such consistency is used for convergence in Section \ref{sec:convergence} and the Delaunay condition plays a role here. 


Consider the transition matrix in Proposition \ref{prop:transit}, the ambiguity of sign is due to the branches of the square root $\lambda = \pm \sqrt{X/\tilde{X}}$. In the following, we shall fix a branch for $\lambda$ in order to define a lift $A:F \to SL(2, \mathbb{C})$. Notice that when two circle patterns are identical $X=\tilde{X}$, it is natural to expect the transition matrix has eigenvalues $\lambda=1$ rather than $\lambda=-1$.

\begin{definition}\label{def:osculatinglift}
	Suppose $z,\tilde{z}:V \to \mathbb{C}$ are two circle patterns with cross ratios $X$ and $\tilde{X}$. The osculating M\"{o}bius transformation from $z$ to $\tilde{z}$ induce a map $\tilde{A}:F \to SL(2,\mathbb{C})/\{\pm I\}$.  We call its lift $A:F \to SL(2,\mathbb{C})$ coherent if the sign over faces is chosen such that the transition matrix satisfies
	\begin{equation}\label{eq:transA}
			A_{jil}^{-1} A_{ijk}=\frac{1}{z_j-z_i}\left(
		\begin{array}{cc}
		\frac{z_j}{\lambda_{ij}} - \lambda_{ij} z_i & -z_i z_j (\frac{1}{\lambda_{ij}}-\lambda_{ij}) \\
		\frac{1}{\lambda_{ij}} - \lambda_{ij} & \lambda_{ij} z_j - \frac{z_i}{\lambda_{ij}}
		\end{array}
		\right)
	\end{equation}
	with
	\begin{equation} \label{eq:lambdasign}
			-\frac{\pi}{2}< \Arg \lambda  \leq \frac{\pi}{2}.
	\end{equation}
\end{definition}

\begin{proposition}
  A coherent lift $A: F \to SL(2,\mathbb{C})$ exists if both circle patterns are Delaunay and the domain is simply connected. Such a lift is unique up to multiplying $-1$ to $A$ over all faces at the same time.
\end{proposition}
\begin{proof}
	The Delaunay condition yields
	\[
	-\frac{\pi}{2} < \frac{\Arg X - \Arg {\tilde{X}}}{2} < \frac{\pi}{2},
	\]
	and hence we could choose
	\[
	\Arg \lambda = \frac{ \Arg X - \Arg {\tilde{X}} }{2}.
	\]
	For each vertex $i$, we indeed have
	\[
	\Arg \prod_{j} \lambda_{ij} = \sum_j \Arg \lambda_{ij} = \frac{\sum _j \Arg X_{ij} - \sum_j \Arg {\tilde{X}_{ij}} }{2} =0 \mod 2\pi
	\]
	Thus multiplying the transition matrices (the right side of Equation \eqref{eq:transA}) around a vertex always yields the identity. It implies whenever the sign of one face is fixed, the sign of neighboring faces are determined by condition \eqref{eq:lambdasign} consistently.
\end{proof}

Notice that such a lift might not exist for non-Delaunay circle patterns. Recall that generally we have $-\pi <\Arg X \leq \pi$ and hence
\[
-\pi < \frac{\Arg X - \Arg {\tilde{X}}}{2} < \pi.
\]
Imposing condition \ref{eq:lambdasign} and the condition that $\lambda^2 = X/\tilde{X}$ yields
\[
\Arg \lambda = \frac{ \Arg X - \Arg {\tilde{X}} }{2} + k \pi
\]
where $k=0, \pm 1$. Thus for a fixed vertex $i$
\[
\Arg \prod_{j} \lambda_{ij} = \sum_j \Arg \lambda_{ij}  =0  \text{ or } \pi   \mod  2\pi
\]
where the product and the sum is over all edges $\{ij\}$ connected to vertex $i$. In case the argument is $\pi$, multiplying the transition matrices (the right side of Equation \eqref{eq:transA}) around vertex $i$ yields $-I$ instead of the identity $I$. It implies no coherent lift exists around the vertex $i$.

\subsection{Horospheres at vertices} \label{sec:horosphersver}
This section illustrates how horospherical nets arise from osculating M\"{o}bius transformations.

Suppose $z:V \to \mathbb{C}$ is a circle pattern. We consider a collection of horospheres $\{H_i\}_{i\in V}$ such that $H_i$ is a horosphere tangent to the Riemann sphere at $z_i$. Given another circle pattern $\tilde{z}:V \to \mathbb{C}$, we are interested in determining whether the osculating M\"{o}bius transformation induces a new collection of horospheres $\{\tilde{H}_i\}_{i\in V}$ adapted to the circle pattern $\tilde{z}$.


	\begin{proposition}
		Suppose $z:V \to \mathbb{C}$ is a circle pattern equipped with a collection of horospheres $\{H_i\}_{i\in V}$ such that $H_i$ is a horosphere tangent to the Riemann sphere at $z_i$. Then a circle pattern $\tilde{z}:V \to \mathbb{C}$ shares the same modulus of cross ratios with $z$, i.e. $|\tilde{X}|=|X|$, if and only if a collection of horospheres $\{\tilde{H}_i\}_{i\in V}$ is induced such that for $i \in V$
		\begin{enumerate}
			\item $\tilde{H}_i$ is a horosphere tangent to the Riemann sphere at $\tilde{z}_i$ and
			\item $\tilde{H}_i = A_{ijk}(H_i)$ for every face $\{ijk\}$ containing vertex $i$.
		\end{enumerate}
%
\end{proposition}
\begin{proof}
	Notice that $A_{ijk}$ is a M\"{o}bius transformation mapping $z_i$ to $\tilde{z}_i$. As an isometry of $\mathbb{H}^3$, it map $H_i$ to a horosphere $A_{ijk}(H_i)$ tangent at $\tilde{z}_i$. It remains to check if $\{ijk\},\{jil\}$ are two neighboring triangles, then the horospheres $A_{ijk}(H_i)= A_{jil}(H_i)$ and $A_{ijk}(H_j)= A_{jil}(H_j)$. It holds if and only if the eigenvalues of $A^{-1}_{jil}A_{ijk}$ have norm $|\lambda_{ij}|=1$, which is equivalent to say $z$ and $\tilde{z}$ share the same modulus of cross ratios $|X| = |\tilde{X}|$ since $|\lambda|=\sqrt{|X/\tilde{X}|}$. 
\end{proof}

To relate horospherical nets in the next section, we consider a collection $\{H_i\}_{i \in V}$ such that $H_i$ is the unique horosphere that passes through the center $O$ in the unit ball model of $\mathbb{H}^3$ and tangent to the Riemann sphere at $z_i$. Thus $O$ is the common intersection point of $\{H_i\}_{i \in V}$. For another circle pattern $\tilde{z}$ sharing the same modulus of cross ratios, the induced horospheres $\{\tilde{H}_i\}_{i \in V}$ generally do not intersect at a common point. Indeed, representing the center $O$ as the identity matrix in the Hermitian matrix model, the M\"{o}bius transformation $A_{ijk}$ maps $O \in H_i \cap H_j \cap H_k$ to the point $A_{ijk} A_{ijk}^* \in \Herm(2)$, which is an intersection point in $\tilde{H}_i \cap \tilde{H}_j \cap \tilde{H}_k$. As shown in Theorem \ref{thm:horo}, the image points form a realization of the dual graph regarded as a discrete CMC-1 surface (See Figure \ref{fig:horocmc1}).

\begin{figure}
\definecolor{uuuuuu}{rgb}{0.26666666666666666,0.26666666666666666,0.26666666666666666}
\begin{tikzpicture}[line cap=round,line join=round,>=triangle 45,x=1.0cm,y=1.0cm,scale=3]
\clip(-1.1078965292064096,-1.1513909442889898) rectangle (3.590595408618125,1.102684171984352);
\draw [line width=0.5pt,dash pattern=on 1pt off 1pt] (0.,0.) circle (1.cm);
\draw [line width=.5pt] (-0.435222047730096,-0.2461336408734533) circle (0.5cm);
\draw [line width=.5pt] (-0.3,-0.4) circle (0.5cm);
\draw [line width=.5pt] (-0.01589643662948944,-0.49974723941457105) circle (0.5cm);
\draw [line width=.5pt] (0.3,-0.4) circle (0.5cm);
\draw [line width=.5pt] (0.4428085726075621,-0.23220802747806407) circle (0.5cm);
\draw (-0.09197118302496266,-1.0248248151006866) node[anchor=north west] {$z_i$};
\draw [line width=0.5pt,dash pattern=on 1pt off 1pt] (2.4,0.) circle (1.cm);
\draw [line width=.5pt] (1.7529854586888014,-0.4323970013086774) circle (0.22180016486195725cm);
\draw [line width=.5pt] (2.010036941768739,-0.67252299573905) circle (0.2225951076929705cm);
\draw [line width=.5pt] (2.4,-0.770597856484146) circle (0.22940214351585397cm);
\draw [line width=.5pt] (2.784334682995496,-0.6649288634770114) circle (0.2319872774243395cm);
\draw [line width=.5pt] (3.0496920066110573,-0.4130158054049395) circle (0.23014172799885246cm);
\draw [shift={(1.7529854586888014,-0.4323970013086774)},line width=2.2pt]  plot[domain=-0.09253133215357146:1.483273048653535,variable=\t]({1.*0.22180016486195628*cos(\t r)+0.*0.22180016486195628*sin(\t r)},{0.*0.22180016486195628*cos(\t r)+1.*0.22180016486195628*sin(\t r)});
\draw [shift={(2.010036941768739,-0.67252299573905)},line width=2.2pt]  plot[domain=0.23568156080607866:1.7341497444097895,variable=\t]({1.*0.22259510769297053*cos(\t r)+0.*0.22259510769297053*sin(\t r)},{0.*0.22259510769297053*cos(\t r)+1.*0.22259510769297053*sin(\t r)});
\draw [shift={(2.4,-0.770597856484146)},line width=2.2pt]  plot[domain=0.7994113226964386:2.428599568063112,variable=\t]({1.*0.22940214351585386*cos(\t r)+0.*0.22940214351585386*sin(\t r)},{0.*0.22940214351585386*cos(\t r)+1.*0.22940214351585386*sin(\t r)});
\draw [shift={(2.784334682995496,-0.6649288634770114)},line width=2.2pt]  plot[domain=1.4134720632834537:2.885335387107162,variable=\t]({1.*0.23198727742434058*cos(\t r)+0.*0.23198727742434058*sin(\t r)},{0.*0.23198727742434058*cos(\t r)+1.*0.23198727742434058*sin(\t r)});
\draw [shift={(3.0496920066110573,-0.4130158054049395)},line width=2.2pt]  plot[domain=1.7999762531034311:3.240784704585702,variable=\t]({1.*0.23014172799885244*cos(\t r)+0.*0.23014172799885244*sin(\t r)},{0.*0.23014172799885244*cos(\t r)+1.*0.23014172799885244*sin(\t r)});
\draw (2.3167797396858024,-1.0248248151006866) node[anchor=north west] {$\tilde{z}_i$};
\draw (0.2154677571989289,-0.4516377180995189) node[anchor=north west] {$H_i$};
\draw (2.3167797396858024,-0.6563434785553244) node[anchor=north west] {$\tilde{H}_i$};
\begin{scriptsize}
\draw [fill=uuuuuu] (0.,0.) circle (1.0pt);
\draw[color=uuuuuu] (-0.01304099865871409,0.1315356692455083) node {$O$};
\draw [fill=black] (-0.870444095460192,-0.4922672817469066) circle (.5pt);
\draw [fill=black] (-0.6,-0.8) circle (.5pt);
\draw [fill=black] (-0.03179287325897888,-0.9994944788291421) circle (.5pt);
\draw [fill=black] (0.6,-0.8) circle (.5pt);
\draw [fill=black] (0.8856171452151242,-0.46441605495612814) circle (.5pt);
\draw [fill=black] (1.5685753297590632,-0.555637487679467) circle (.5pt);
\draw [fill=black] (1.8983784356257327,-0.8650871667957587) circle (.5pt);
\draw [fill=black] (2.4,-1.) circle (.5pt);
\draw [fill=black] (2.9004274951417006,-0.8657784486265544) circle (.5pt);
\draw [fill=black] (3.243911185005867,-0.5364829091611342) circle (.5pt);
\draw [fill=black] (1.772373361133687,-0.21144582486560207) circle (1.0pt);
\draw [fill=black] (1.9738367690863194,-0.45289119142417406) circle (1.0pt);
\draw [fill=black] (2.2264785009353436,-0.6205457561588806) circle (1.0pt);
\draw [fill=black] (2.5599228593451655,-0.6061289460222147) circle (1.0pt);
\draw [fill=black] (2.8206815400877883,-0.4358066191045576) circle (1.0pt);
\draw [fill=black] (2.9974086461733176,-0.18889158719629706) circle (1.0pt);
\end{scriptsize}
\end{tikzpicture}
\caption{The horospheres $\{H_i\}$ on the left passes through a common vertex $O$, the center of the unit ball model. The osculating M\"{o}bius transformation induces horospheres $\{\tilde{H}_i\}$ adapted to another circle pattern. As seen in Section \ref{sec:weier}, the thick arcs form a cross section of a discrete CMC-1 surface.}
\label{fig:horocmc1}
\end{figure}
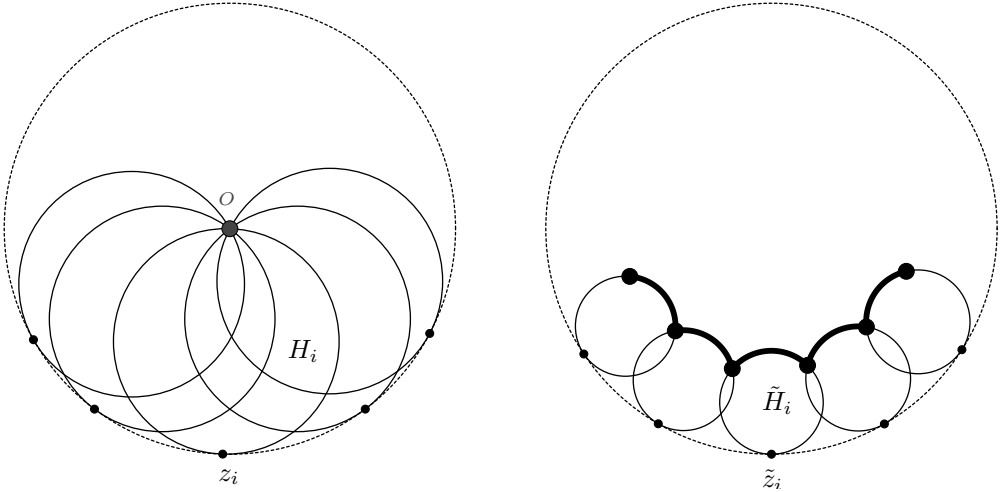

Generally, one could pick any horosphere at vertices, e.g. horospheres that are tangent to a prescribed totally geodesic plane or tangent to a prescribed horosphere, and consider their image under an osculating M\"{o}bius transformation. The image horospheres shall enclose many surfaces of interest. Such choices of horospheres are reminiscent of the smooth theory: any conformal metric $\rho_{\infty}$ on a domain $\Omega \subset \partial \mathbb{H}^3$ yields a collection of horospheres $\{H_w\}_{w\in\Omega}$ by considering visual metrics. For more information on the smooth theory, see the Epstein map in  \cite[Section 3]{Anderson1998}.


%
 
\section{Discrete CMC-1 surfaces in hyperbolic space} \label{sec:discretecmc}

In this section, we define a notion of integrated mean curvature on horospherical nets. It is motivated by Steiner's formula relating curvatures to the change in area under parallel surfaces. The main result is to prove the correspondence between discrete CMC-1 surfaces and circle patterns  (Theorem \ref{thm:horo}).

Recall that we denote by $(V,E,F)$ a cell decomposition of a surface, where $V,E,F$ are the sets of (primal) vertices, edges and faces. It is associated with a dual cell decomposition $(V^*,E^*,F^*)$,  where $V^*,E^*,F^*$ denote the sets of dual vertices, dual edges and dual faces. Each dual vertex corresponds to a primal face, whereas each dual face corresponds to a primal vertex. By subdivision, we assume the primal cell decomposition is a triangulation without loss of generality. The dual mesh then becomes a trivalent mesh, i.e. each vertex has three neighboring vertices. Furthermore, for simplicity we consider surfaces without boundary. In case the surface has boundary, the same result also holds but requires modification of notations.

The following choice of notations is taken so that it will be consistent with the osculating M\"{o}bius transformations.
\begin{definition}\label{def:horo}
	A horospherical net is a realization $f:V^{*} \to \mathbb{H}^{3}$ of a dual mesh such 
	that
	\begin{enumerate}
	   \item  The vertices of each dual face lies on the same horosphere. 
	   \item Every edge is realized as the shorter circular arc in the intersection of the two neighboring horospheres. The hyperbolic lengths of the arcs are denoted as $\ell$.
	   \item  The tangency points of the horospheres with $\partial \mathbb{H}^3$ define a realization of the primal triangular mesh $\tilde{z}:V \to \mathbb{C} \cup \{\infty\}$ that forms a Delaunay circle pattern. We call $\tilde{z}$ the hyperbolic Gauss map.
	\end{enumerate} 
\end{definition}

We further define the \textit{dihedral angle} $\alpha:E^* \to (-\pi,\pi)$ between neighboring horospheres. Its sign is defined as follows: consider a dual edge oriented from dual vertex $u$ to $v$ and let $x$ be a point on the dual edge under $f$. It is associated with three unit vectors $(T_{uv},N_l,N_r)$ in $T_{x}\mathbb{H}^3$,  where $T_{uv}$ is the unit tangent vector of the circular arc oriented from $u$ to $v$ while $N_l, N_r$ are the normals of the horospheres from the left and the right face which are oriented toward the points of tangency with $\partial \mathbb{H}^3$. The angle $\alpha$ is determined such that
\begin{align*}
\sin \alpha_{uv} &= \langle N_l \times N_r , T_{uv} \rangle \\
\cos \alpha_{uv} &= \langle N_l,  N_r \rangle
\end{align*}
In particular, the horospheres coincide if $\alpha=0$.

\subsection{Integrated mean curvature via parallel surfaces}

In this subsection, we introduce integrated mean curvature over the faces of a horospherical net. Our approach is motivated by the curvature theory of polyhedral surfaces in $\mathbb{R}^3$ by Bobenko-Pottmann-Wallner \cite{Bobenko2010a}.

Recall that in the smooth theory, the Steiner formula relates the mean curvature and Gaussian curvature of a smooth surface to the area of its parallel surface. We denote by $f:\Omega \subset \mathbb{C} \to \mathbb{H}^3$ a smooth immersion and $N$ its unit normal vector field. For small $t$, we write $f_t: \Omega \to \mathbb{H}^3$ the mapping such that for every $x \in \Omega$, the geodesic starting at $f(x)$ in the direction of $N(x)$ with length $t$ ends at $f_t(x)$. We call $f_t$ a parallel surface of $f$ at distance $t$. For each $t$, we denote by $\omega_t$ the area $2$-form induced from the hyperbolic metric in $\mathbb{H}^3$ via $f_t$. By direction computation (see \cite[Section 4]{Epstein1984})
\begin{align}
\frac{d}{dt} \, \omega_t |_{t=0}& = -2H \omega \label{eq:Steinerm}\\
\frac{d^2}{dt^2} \, \omega_t|_{t=0} &= (2K + 4) \omega  \label{eq:Steinerk}
\end{align}
where $H$ and $K$ is the mean curvature and Gaussian curvature of $f$.

\begin{example}
	Let $f$ be a horosphere and $N$ the unit normal pointing toward the tangency point of the horosphere with $\partial \mathbb{H}^3$. We denote by $f_t$ the parallel surfaces which remain as horospheres for all $t$. Using the upper half space model, we have the area 2-form 
	\begin{align*}
	\omega_t = e^{-2t} \omega = (1-2t +\frac{(2t)^2}{2} +\dots) \omega
	\end{align*}
	Compared with \eqref{eq:Steinerm} and \eqref{eq:Steinerk}, it shows that the horosphere has constant mean curvature $H\equiv 1$ and constant Gaussian curvature $K\equiv 0$.
\end{example}

Similarly we consider parallel surfaces for a trivalent horospherical net $f$. Here we assume the edges have non-zero lengths. Every vertex is the intersection point of the three neighboring horospheres. For every small $t$, the parallel surfaces of the three horospheres at distance $t$ are again horospheres and intersect. The intersection points continuously define a new mapping  $f_t$ of the trivalent mesh. By construction, $f_t$ is a horospherical net.

We compute the change in face area using the upper half space model. Let $\phi$ be a face and we normalize the horosphere containing $f(\phi)$ as the horizontal plane $x_3=1$. Then $f_{t}(\phi)$ is contained in the plane $x_3=e^{t}$. We denote by $\widetilde{\area}(f_{t}(\phi))$ the Euclidean area on the horosphere $x_3=e^{t}$. Then its hyperbolic area is
\[
\Area(f_{t}(\phi)) = e^{-2t} \widetilde{\area}(f_{t}(\phi)).
\]
The change in area can be expressed in terms of edge lengths and dihedral angles, which motivates the definition of integrate mean curvature in Definition \ref{def:intmean}.

\begin{proposition}
	\[
	\frac{d}{dt} \area(f_t(\phi)) |_{t=0} = -2 \area(f(\phi)) -  \sum_{ij \in \phi} \ell_{ij} \tan \frac{\alpha_{ij}}{2}
	\]
\end{proposition}
\begin{proof}
	Notice that for all $t$ the tangency points of the horospheres with $\partial \mathbb{H}^3$ remain unchanged. 	We denote by $\tilde{r}_t$ the Euclidean radii of the circular edges of $f_{t}(\phi)$ on the horosphere $x_3=e^{t}$ (See Figure \ref{fig:parallel}). By trigonometry, one obtains
	\[
	\tilde{r}_t = \sqrt{\frac{1}{\sin^2 \frac{\alpha}{2}} - e^{2t}}
	\]
    and 
	\[
	\frac{d}{dt}\tilde{r}_t |_{t=0}  = -|\tan \frac{\alpha_{ij}}{2}|.
	\]
	Having the centers (which are independent of $t$) and the radii $\tilde{r}_t$, the vertices of $f_t(\phi)$ can be determined explicitly in terms of $t$. One can show that the Euclidean area
	\[
	\widetilde{\area}(f_{t}(\phi))-  \widetilde{\area}(f(\phi)) = \sum  \sign(\alpha) \ell \Delta \tilde{r}  + O(\sum (\Delta r)^2) 
	\]
	where $\Delta \tilde{r}= \tilde{r}_t - \tilde{r}_0$ is the change of the Euclidean radii (See Figure \ref{fig:parallel} right). The second order term is contributed by the change in area at the corners (See Figure \ref{fig:corner}). Differentiating both sides with respect to $t$ at $t=0$ yields
	\[
		\frac{d}{dt} \widetilde{\area}(f_{t}(\phi)) |_{t=0} = \sum_{ij \in \phi} \sign(\alpha) \, \ell_{ij} \frac{d}{dt}\tilde{r}_t |_{t=0} = - \sum_{ij \in \phi} \ell_{ij} \tan \frac{\alpha_{ij}}{2}
	\]
	Thus
	\begin{align*}
\frac{d}{dt}  \Area(f_{t}(\phi))|_{t=0} &= 	\frac{d}{dt}  \left( e^{-2t} \widetilde{\area}(f_{t}(\phi)) \right) |_{t=0} \\
&=-2 \area(f(\phi)) -  \sum_{ij \in \phi} \ell_{ij} \tan \frac{\alpha_{ij}}{2}
	\end{align*}
	since $\area(f(\phi)) = \widetilde{\area}(f_{0}(\phi))$.
\end{proof}

%
%
%

\begin{definition} \label{def:discretecmc}
	A horospherical net $f$ is a discrete constant mean curvature-1 surface if it satisfies the following two properties:
	\begin{enumerate}
	\item For every face $\phi$, the ratio of the integrated mean curvature to the face area is constantly equal to 1, i.e.
	\[
	\frac{H_{\phi}}{\area(f(\phi))} = 1.
	\]
	\item For every edge
		\[
	0 \leq \ell \tan \frac{\alpha}{2} + \Arg \tilde{X} < \pi.
	\]
	where $\tilde{X}$ is the cross ratio of the hyperbolic Gauss map $\tilde{z}$.
	\end{enumerate}
\end{definition}
As we shall see in the next section, a discrete CMC-1 surface corresponds to a pair of circle patterns $z$ and $\tilde{z}$, one of which $\tilde{z}$ is the hyperbolic Gauss map and is already assumed to be Delaunay in Definition \eqref{def:horo}. Condition (2) in Definition \ref{def:discretecmc} is then equivalent to the other circle pattern $z$ being Delaunay.

\begin{figure}
	\begin{minipage}{.55\textwidth}
\begin{tikzpicture}[line cap=round,line join=round,>=triangle 45,x=1.0cm,y=1.0cm]
\clip(-1.5990187908523654,0.574431364951998) rectangle (3.736087504308259,5.022810241010297);
\draw [shift={(2.7320508075688763,2.)},line width=.5pt,fill=black,fill opacity=0.10000000149011612] (0,0) -- (0.:0.3353930405446931) arc (0.:60.:0.3353930405446931) -- cycle;
\draw [line width=.5pt,dash pattern=on 3pt off 3pt,domain=-3.5990187908523654:5.736087504308259] plot(\x,{(--8.-0.*\x)/8.});
\draw [line width=0.5pt,domain=-3.5990187908523654:5.736087504308259] plot(\x,{(--16.-0.*\x)/8.});
\draw [line width=0.5pt] (1.,3.) circle (2.cm);
\draw [line width=.5pt,dash pattern=on 3pt off 3pt] (2.732050807568877,2.)-- (3.317050807568877,3.013249722427793);
\draw [line width=.5pt,dash pattern=on 3pt off 3pt] (1.,3.)-- (1.,1.);
\draw [line width=0.5pt,dash pattern=on 3pt off 3pt] (1.,3.)-- (2.732050807568877,2.);
\draw (1.6890114817356288,1.8173637296932556) node[anchor=north west] {$\tilde{r}$};
\draw (3.0976622520233397,2.531633058528451) node[anchor=north west] {$\alpha$};
		\draw (-0.4783922768040664,2.486296074292418) node[anchor=north west] {$x_3=1$};
\draw [->,line width=0.5pt] (1.,1.9025874104060527) -- (2.752237768781757,1.9025874104060527);
\draw [->,line width=0.5pt] (2.752237768781757,1.9025874104060527) -- (1.,1.9025874104060527);
\end{tikzpicture}
\end{minipage}
\begin{minipage}{.4\textwidth}
\begin{tikzpicture}[line cap=round,line join=round,>=triangle 45,x=1.0cm,y=1.0cm,scale=0.8]
\clip(-1.6252784909831166,-1.212826542490937) rectangle (4.460476880423979,4.107206680250292);
\draw [shift={(0.,4.)},line width=.5pt,fill=black,fill opacity=0.10000000149011612] (0,0) -- (-93.61388075200365:0.42857355369347044) arc (-93.61388075200365:-64.91570791629776:0.42857355369347044) -- cycle;
\draw [shift={(0.,4.)},line width=.5pt]  plot[domain=4.649314862487778:5.150192467777242,variable=\t]({1.*3.807571404451924*cos(\t r)+0.*3.807571404451924*sin(\t r)},{0.*3.807571404451924*cos(\t r)+1.*3.807571404451924*sin(\t r)});
\draw [shift={(3.,4.)},line width=.5pt]  plot[domain=4.331849746950265:4.778344288593576,variable=\t]({1.*3.7219210282432544*cos(\t r)+0.*3.7219210282432544*sin(\t r)},{0.*3.7219210282432544*cos(\t r)+1.*3.7219210282432544*sin(\t r)});
\draw [shift={(4.36,1.38)},line width=.5pt]  plot[domain=2.2835950364618505:3.923965332279361,variable=\t]({1.*1.5597435686676198*cos(\t r)+0.*1.5597435686676198*sin(\t r)},{0.*1.5597435686676198*cos(\t r)+1.*1.5597435686676198*sin(\t r)});
\draw [shift={(3.,-1.)},line width=.5pt]  plot[domain=1.4755795091115858:1.9742153600947085,variable=\t]({1.*3.5761990996028175*cos(\t r)+0.*3.5761990996028175*sin(\t r)},{0.*3.5761990996028175*cos(\t r)+1.*3.5761990996028175*sin(\t r)});
\draw [shift={(0.,-1.)},line width=.5pt]  plot[domain=1.1169523252733864:1.6475682180646747,variable=\t]({1.*3.649438312946254*cos(\t r)+0.*3.649438312946254*sin(\t r)},{0.*3.649438312946254*cos(\t r)+1.*3.649438312946254*sin(\t r)});
\draw [shift={(-1.5031088116251143,1.4037211435566217)},line width=.5pt]  plot[domain=-0.7613283248537828:0.7907531219112229,variable=\t]({1.*1.744817601197979*cos(\t r)+0.*1.744817601197979*sin(\t r)},{0.*1.744817601197979*cos(\t r)+1.*1.744817601197979*sin(\t r)});
\draw [line width=0.3pt,dash pattern=on 3pt off 3pt] (0.,4.)-- (-0.24,0.2);
\draw [line width=0.3pt,dash pattern=on 3pt off 3pt] (0.,4.)-- (1.617599690582241,0.5443284986849509);
\draw [line width=0.3pt,dash pattern=on 3pt off 3pt] (3.34,2.56)-- (4.36,1.38);
\draw [line width=0.3pt,dash pattern=on 3pt off 3pt] (4.36,1.38)-- (3.246179194617429,0.272898581345754);
\draw [line width=0.3pt,dash pattern=on 3pt off 3pt] (1.617599690582241,0.5443284986849509)-- (3.,4.);
\draw [line width=0.3pt,dash pattern=on 3pt off 3pt] (3.,4.)-- (3.246179194617429,0.272898581345754);
\draw (0.7318810129490687,0.3586137879462571) node[anchor=north west] {$\ell$};
\draw [->,line width=0.3pt] (0.11759225265509433,4.04434585384828) -- (1.7176001864440509,0.6300432094236448);
\draw [->,line width=0.3pt] (1.7176001864440509,0.6300432094236447) -- (0.11759225265509432,4.04434585384828);
\draw (-0.10955138369917818,3.6586296555241575) node[anchor=north west] {$\theta$};
\draw (1.114740351765655,2.2729084985819417) node[anchor=north west] {$\tilde{r}$};
\begin{scriptsize}
\draw [fill=black] (0.,4.) ++(-2.0pt,0 pt) -- ++(2.0pt,2.0pt)--++(2.0pt,-2.0pt)--++(-2.0pt,-2.0pt)--++(-2.0pt,2.0pt);
\draw [fill=black] (3.,4.) ++(-2.0pt,0 pt) -- ++(2.0pt,2.0pt)--++(2.0pt,-2.0pt)--++(-2.0pt,-2.0pt)--++(-2.0pt,2.0pt);
\draw [fill=black] (-1.5031088116251143,1.4037211435566217) ++(-2.0pt,0 pt) -- ++(2.0pt,2.0pt)--++(2.0pt,-2.0pt)--++(-2.0pt,-2.0pt)--++(-2.0pt,2.0pt);
\draw [fill=black] (0.,-1.) ++(-2.0pt,0 pt) -- ++(2.0pt,2.0pt)--++(2.0pt,-2.0pt)--++(-2.0pt,-2.0pt)--++(-2.0pt,2.0pt);
\draw [fill=black] (3.,-1.) ++(-2.0pt,0 pt) -- ++(2.0pt,2.0pt)--++(2.0pt,-2.0pt)--++(-2.0pt,-2.0pt)--++(-2.0pt,2.0pt);
\draw [fill=black] (4.36,1.38) ++(-2.0pt,0 pt) -- ++(2.0pt,2.0pt)--++(2.0pt,-2.0pt)--++(-2.0pt,-2.0pt)--++(-2.0pt,2.0pt);
\end{scriptsize}
\end{tikzpicture}
\end{minipage}

	\caption{On the left, the solid line and the circle indicates a vertical cross section of two horospheres intersect at angle $|\alpha|$ in the upper half space model, where one of the horospheres is the horizontal plane $x_3=1$. 
		The right figure shows the horospherical face on the horizontal plane $x_3=1$. The rhombi vertices denote the projection of the tangency points of the neighboring horospheres with $\partial \mathbb{H}^3$. Each circular edge is generated by a rotation centered at a rhombus vertex with angle $\theta$ and radius $\tilde{r}$. The rotation is clockwise if $\theta>0$ while counterclockwise if $\theta<0$. \medskip} \label{fig:horosphere}
		\begin{minipage}{.55\textwidth}
	\begin{tikzpicture}[line cap=round,line join=round,>=triangle 45,x=1.0cm,y=1.0cm]
		\clip(-1.5990187908523654,0.464431364951998) rectangle (5.106087504308259,5.022810241010297);
	\draw[line width=.5pt,fill=black,fill opacity=0.10000000149011612] (3.237158693331927,1.) -- (3.237158693331927,1.2371586933319272) -- (3.,1.2371586933319272) -- (3.,1.) -- cycle; 
	\draw [shift={(1.,1.)},line width=.5pt,fill=black,fill opacity=0.10000000149011612] (0,0) -- (30.:0.3353930405446931) arc (30.:90.:0.3353930405446931) -- cycle;
	\draw [line width=.5pt,dash pattern=on 3pt off 3pt,domain=-3.7219962390520864:5.613110056108538] plot(\x,{(--8.-0.*\x)/8.});
	\draw [line width=.5pt,domain=-3.7219962390520864:5.613110056108538] plot(\x,{(--16.-0.*\x)/8.});
	\draw [line width=.5pt] (1.,3.) circle (2.cm);
	\draw [line width=.5pt,dash pattern=on 3pt off 3pt] (1.,3.)-- (1.,1.);
	\draw (1.4654161213725,2.509891434655547) node[anchor=north west] {$\tilde{r}_t$};
	\draw (3.2206397002230607,2.438655610328731) node[anchor=north west] {$x_3=1$};
	\draw [shift={(1.,1.)},line width=.5pt,dotted]  plot[domain=0.:3.141592653589793,variable=\t]({1.*2.*cos(\t r)+0.*2.*sin(\t r)},{0.*2.*cos(\t r)+1.*2.*sin(\t r)});
	\draw [line width=.5pt,dash pattern=on 3pt off 3pt] (1.,1.)-- (2.732050807568877,2.);
	\draw (0.8617086483920525,1.0782632161681163) node[anchor=north west] {$\frac{\pi-\alpha}{2}$};
	\draw [line width=.5pt,dash pattern=on 3pt off 3pt] (1.,2.252138858074541) circle (1.2521388580745412cm);
	\draw [line width=.5pt,dash pattern=on 3pt off 3pt,domain=-3.7219962390520864:5.613110056108538] plot(\x,{(--6.252325188263848-0.*\x)/2.4072709263645913});
	\draw (3.231819468241217,3.1570821553818037) node[anchor=north west] {$x_3=e^{t}$};
	\draw [->,line width=.5pt] (1.,2.47573421843767) -- (2.203280810570825,2.4869139864558263);
	\draw [->,line width=.5pt] (2.203280810570825,2.4869139864558263) -- (1.,2.47573421843767);
	\end{tikzpicture}
\end{minipage}
		\begin{minipage}{.4\textwidth}
	\begin{tikzpicture}[line cap=round,line join=round,>=triangle 45,x=1.0cm,y=1.0cm,scale=0.8]
\clip(-1.6252784909831166,-1.212826542490937) rectangle (4.460476880423979,4.107206680250292);
	\draw [shift={(0.,4.)},line width=.5pt]  plot[domain=4.649314862487778:5.150192467777242,variable=\t]({1.*3.807571404451924*cos(\t r)+0.*3.807571404451924*sin(\t r)},{0.*3.807571404451924*cos(\t r)+1.*3.807571404451924*sin(\t r)});
	\draw [shift={(3.,4.)},line width=.5pt]  plot[domain=4.331849746950265:4.778344288593576,variable=\t]({1.*3.7219210282432544*cos(\t r)+0.*3.7219210282432544*sin(\t r)},{0.*3.7219210282432544*cos(\t r)+1.*3.7219210282432544*sin(\t r)});
	\draw [shift={(4.36,1.38)},line width=.5pt]  plot[domain=2.2835950364618505:3.923965332279361,variable=\t]({1.*1.5597435686676198*cos(\t r)+0.*1.5597435686676198*sin(\t r)},{0.*1.5597435686676198*cos(\t r)+1.*1.5597435686676198*sin(\t r)});
	\draw [shift={(3.,-1.)},line width=.5pt]  plot[domain=1.4755795091115858:1.9742153600947085,variable=\t]({1.*3.5761990996028175*cos(\t r)+0.*3.5761990996028175*sin(\t r)},{0.*3.5761990996028175*cos(\t r)+1.*3.5761990996028175*sin(\t r)});
	\draw [shift={(0.,-1.)},line width=.5pt]  plot[domain=1.1169523252733864:1.6475682180646747,variable=\t]({1.*3.649438312946254*cos(\t r)+0.*3.649438312946254*sin(\t r)},{0.*3.649438312946254*cos(\t r)+1.*3.649438312946254*sin(\t r)});
	\draw [shift={(-1.5031088116251143,1.4037211435566217)},line width=.5pt]  plot[domain=-0.7613283248537828:0.7907531219112229,variable=\t]({1.*1.744817601197979*cos(\t r)+0.*1.744817601197979*sin(\t r)},{0.*1.744817601197979*cos(\t r)+1.*1.744817601197979*sin(\t r)});
	\draw [shift={(-1.5031088116251143,1.4037211435566217)},line width=.5pt,dash pattern=on 3pt off 3pt]  plot[domain=-0.6370474265111508:0.6994297415592061,variable=\t]({1.*1.5300781427268437*cos(\t r)+0.*1.5300781427268437*sin(\t r)},{0.*1.5300781427268437*cos(\t r)+1.*1.5300781427268437*sin(\t r)});
	\draw [shift={(0.,4.)},line width=.5pt,dash pattern=on 3pt off 3pt]  plot[domain=4.634646568354957:5.180672468246205,variable=\t]({1.*3.5170298740520005*cos(\t r)+0.*3.5170298740520005*sin(\t r)},{0.*3.5170298740520005*cos(\t r)+1.*3.5170298740520005*sin(\t r)});
	\draw [shift={(3.,4.)},line width=.5pt,dash pattern=on 3pt off 3pt]  plot[domain=4.292946760114922:4.806087259461275,variable=\t]({1.*3.4513937392926173*cos(\t r)+0.*3.4513937392926173*sin(\t r)},{0.*3.4513937392926173*cos(\t r)+1.*3.4513937392926173*sin(\t r)});
	\draw [shift={(4.36,1.38)},line width=.5pt,dash pattern=on 3pt off 3pt]  plot[domain=2.380769200913346:3.805207411421898,variable=\t]({1.*1.3139965436721635*cos(\t r)+0.*1.3139965436721635*sin(\t r)},{0.*1.3139965436721635*cos(\t r)+1.*1.3139965436721635*sin(\t r)});
	\draw [shift={(3.,-1.)},line width=.5pt,dash pattern=on 3pt off 3pt]  plot[domain=1.4471723083194692:2.0109861200044277,variable=\t]({1.*3.311295258203884*cos(\t r)+0.*3.311295258203884*sin(\t r)},{0.*3.311295258203884*cos(\t r)+1.*3.311295258203884*sin(\t r)});
	\draw [shift={(0.,-1.)},line width=.5pt,dash pattern=on 3pt off 3pt]  plot[domain=1.0860395754373826:1.6668097799409811,variable=\t]({1.*3.398511387500966*cos(\t r)+0.*3.398511387500966*sin(\t r)},{0.*3.398511387500966*cos(\t r)+1.*3.398511387500966*sin(\t r)});
	\draw [->,line width=.5pt] (0.,4.) -- (1.5944161474823686,0.8477860514917506);
	\draw [->,line width=.5pt] (1.5944161474823686,0.8477860514917507) -- (0.,4.);
	\draw (0.6713676722096803,3.1661403614789596) node[anchor=north west] {$\tilde{r}_t$};
	\begin{scriptsize}
	\draw [fill=black] (0.,4.) ++(-2.0pt,0 pt) -- ++(2.0pt,2.0pt)--++(2.0pt,-2.0pt)--++(-2.0pt,-2.0pt)--++(-2.0pt,2.0pt);
	\draw [fill=black] (3.,4.) ++(-2.0pt,0 pt) -- ++(2.0pt,2.0pt)--++(2.0pt,-2.0pt)--++(-2.0pt,-2.0pt)--++(-2.0pt,2.0pt);
	\draw [fill=black] (-1.5031088116251143,1.4037211435566217) ++(-2.0pt,0 pt) -- ++(2.0pt,2.0pt)--++(2.0pt,-2.0pt)--++(-2.0pt,-2.0pt)--++(-2.0pt,2.0pt);
	\draw [fill=black] (0.,-1.) ++(-2.0pt,0 pt) -- ++(2.0pt,2.0pt)--++(2.0pt,-2.0pt)--++(-2.0pt,-2.0pt)--++(-2.0pt,2.0pt);
	\draw [fill=black] (3.,-1.) ++(-2.0pt,0 pt) -- ++(2.0pt,2.0pt)--++(2.0pt,-2.0pt)--++(-2.0pt,-2.0pt)--++(-2.0pt,2.0pt);
	\draw [fill=black] (4.36,1.38) ++(-2.0pt,0 pt) -- ++(2.0pt,2.0pt)--++(2.0pt,-2.0pt)--++(-2.0pt,-2.0pt)--++(-2.0pt,2.0pt);
	\end{scriptsize}
	\end{tikzpicture}
\end{minipage}
\caption{The left figure shows the parallel surfaces of two horospheres at hyperbolic distance $t$. By varying $t>0$, the intersection point of the parallel surfaces lie on a totally geodesic plane (the dotted half circle orthogonal to $x_3=1$). The right figure indicates a deformation of the face under parallel surfaces.}
\label{fig:parallel}
\end{figure}

\begin{figure}
	\centering
	\includegraphics[width=0.4\textwidth]{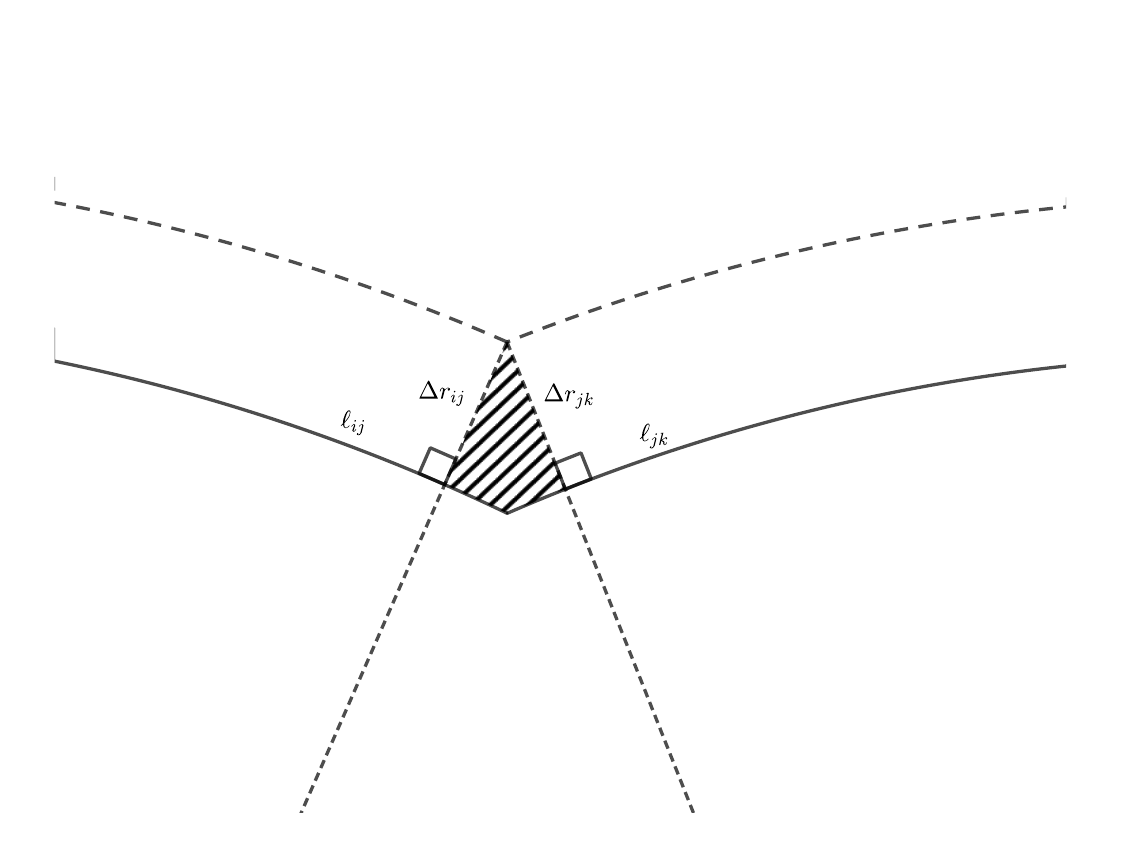}
	\caption{The area of the shaded region is bounded above by $C (\Delta r_{ij}^2+ \Delta r_{jk}^2)$ for some constant $C>0$ independent of $t$.}
	\label{fig:corner}
\end{figure}

\subsection{Weierstrass representation} \label{sec:weier}

We prove a Weierstrass-type representation that every discrete CMC-1 surface corresponds to a pair of Delaunay circle patterns sharing the modulus of cross ratios. It makes use of the discrete osculating M\"{o}bius transformation.  

We first describe horospheres in terms of Hermitian matrices. 
\begin{lemma}\label{lem:horosphere}
	For any $r>0$, the Hermitian matrix
	\[
	N_{z,r}:= \frac{2r}{1+|z|^2} \left( \begin{array}{cc}
	|z|^2 & z \\ \bar{z} & 1
	\end{array}\right) 
	\]
	represents a point in the upper light cone $L^{+}$. Its corresponding horosphere touches the boundary $\partial\mathbb{H}^{3}$ at $z$ in the upper half space model. Furthermore, if $(z,1)^T$ is an eigenvector of $A \in SL(2,\mathbb{C})$ with eigenvalue $\lambda$, then
	\begin{equation}\label{eq:horon}
	A N_{z,r} A^* = |\lambda|^2 N_{z,r}  
	\end{equation}
\end{lemma}
\begin{proof}
	It is obvious that $\det N_{z,r}=0$ and $\trace	N_{z,r}>0$. Thus it represents an element in $L^{+}$. The extended complex plane is identified with the light cone as follows
	\begin{align*}
	\mathbb{C}\cup \{\infty\} &\to L^+ \subset \mathbb{R}^{3,1}\to \Herm(2) \\
	z &\mapsto (1, \frac{2 \Re z}{1 + |z|^2},\frac{2 \Im z}{1 + |z|^2}, -\frac{1-|z|^2}{1 + |z|^2}) \mapsto \frac{2}{1+|z|^2} \left( \begin{array}{cc}
	|z|^2 & z \\ \bar{z} & 1
	\end{array}\right) 
	\end{align*}	
	which is a composition of a stereographic projection and the identification between $\mathbb{R}^{3,1}$ and hermitian matrices. Hence $N_{z,r}$ represents a horosphere touching $\partial H^{3}$ at $z$. Equation \eqref{eq:horon} follows from the observation that the columns of $N_{z,r}$ are multiples of $(z,1)^T$ while the rows are multiplies of $(\bar{z},1)$.
\end{proof}

Recall that each dual vertex corresponds to a primal face and we have a bijection $V^{*} \cong F$. We then prove the main result.
%
%

\begin{proof}[Proof of Theorem \ref{thm:horo}]
	Suppose $A:F \to SL(2,\mathbb{C})$ is the osculating M\"{o}bius transformation mapping $z$ to $\tilde{z}$. Let $\{ijk\},\{jil\}$ be two neighbouring faces. We have
	\[
	f_{jil} = A_{jil} A^*_{jil} = ( A_{jil}A_{ijk}^{-1}) f_{ijk}(A_{jil}A_{ijk}^{-1})^* 
	\]
	Then we have for any $r>0$,
	\begin{align*}
	\langle f_{jil}, \, N_{\tilde{z}_i,r} \rangle &= \langle A_{jil} A^*_{jil}, \, N_{\tilde{z}_i,r} \rangle  \\&= \langle A_{ijk} A^*_{ijk}, \, ( A_{ijk}A_{jil}^{-1}) N_{\tilde{z}_i,r} (A_{ijk}A_{jil}^{-1})^* \rangle\\ &= \langle  A_{ijk} A^*_{ijk}, \,   N_{\tilde{z}_i,r} \rangle = \langle f_{ijk}, \, N_{\tilde{z}_i,r} \rangle
	\end{align*}
	which follows from Lemma \ref{lem:horosphere} and the property that $(\tilde{z}_i,1)^T$ is an eigenvector of $A_{ijk}A_{jil}^{-1}$ with eigenvalue $|X/\tilde{X}|=1$. Hence all the vertices within a dual face $i$ lie on a common horosphere touching $\partial \mathbb{H}^3$ at $\tilde{z}_i$. 
	
	We pick a dual face $i$ and consider the upper half space model normalized in such a way that the horosphere at $\tilde{z}_i$ becomes the horizontal plane $x_3=1$ and $\tilde{z}_i$ becomes infinity. The transition matrix $A_{jil}A_{ijk}^{-1}$ across edge $\{ij\}$ hence becomes a rotation around a vertical axis through  $\tilde{z}_j$. The edge $\{ij\}$ under $f$ is a circular arc on the horizontal plane generated by rotation with radius $r_{ij} \geq 0$ (See Figure \ref{fig:horosphere}). We denote the rotation angle by $\theta_{ij}$ which is related to the eigenvalue $\lambda_{ij}$ of $A_{jil}A_{ijk}^{-1}$ at $\tilde{z}_i$ via
	\begin{equation}\label{eq:theta}
			\theta_{ij} =  \log \lambda^2_{ij}= \Arg X_{ij} - \Arg \tilde{X}_{ij}
	\end{equation}
    and hence satisfies 
    \[
    -\pi <  -\Arg \tilde{X}_{ij} \leq \theta_{ij} \leq \Arg X_{ij} < \pi.
    \]
    Thus the edge is the shorter circular arc in the intersection of the neighboring horospheres. It implies $f$ is a horospherical net.
		
	By trigonometry, one can show that $r_{ij}= |\cot \frac{\alpha_{ij}}{2}|$ and the hyperbolic length of the circular arc is
		\[
	\ell_{ij} =  |\theta_{ij}| r_{ij} =  |\theta_{ij}|  |\cot \frac{\alpha_{ij}}{2}| .
	\] 
    The sign of $\alpha$ is indeed defined such that
    	\begin{equation} \label{eq:ell}
   \ell_{ij} = \theta_{ij}  \cot \frac{\alpha_{ij}}{2} .
    \end{equation}
    Combining with Equation \eqref{eq:theta}, it yields
    \[
    0\leq \Arg X_{ij} = \ell_{ij}  \tan \frac{\alpha_{ij}}{2} + \Arg \tilde{X}_{ij} < \pi
    \]
    as required in Definition \ref{def:discretecmc}. 
    
    It remains to check the integrated mean curvature of $f$. Notice that Equation \eqref{eq:theta} and \eqref{eq:ell} yield
	\[	
	\sum_j \ell_{ij} \tan \frac{\alpha_{ij}}{2} = \sum_j    \theta_{ij}  = 2\pi - 2\pi =0.
	\]
	Hence the integrated mean curvature $H: F^* \to \mathbb{R}$ satisfies for every face $\phi \in F^*$
	\[
	\frac{H_{\phi}}{\area(f(\phi))} = \frac{\area(f(\phi)) + 0}{\area(f(\phi))} = 1
	\]
	and $f$ is a discrete CMC-1 surface.
	
	Conversely, suppose $f:V^{*} \to \mathbb{H}^3$ is a discrete CMC-1 surface with hyperbolic Gauss map $\tilde{z}$. Motivated from Equation \eqref{eq:transit}, we define a function $\eta:\vec{E} \to SL(2,\mathbb{C})$ on oriented edges via
	\begin{equation} \label{eq:eta}
		\eta_{ij}:=\frac{1}{\tilde{z}_j-\tilde{z}_i}\left(
	\begin{array}{cc}
	\lambda_{ij} \tilde{z}_j  - \frac{\tilde{z}_i}{\lambda_{ij} } & -\tilde{z}_i \tilde{z}_j (\lambda_{ij}- \frac{1}{\lambda_{ij}}) \\ \lambda_{ij} -\frac{1}{\lambda_{ij}}  & \frac{ \tilde{z}_j }{\lambda_{ij}}-  \lambda_{ij} \tilde{z}_i 
	\end{array}
	\right)
	\end{equation}
	where $\lambda_{ij}= e^{\mathbf{i} \frac{\ell_{ij}}{2} \tan \frac{\alpha_{ij}}{2}}$. We use the Hermitian matrix model of the hyperbolic space. Particularly, $\eta$ is defined in such a way that
	\[
	f_{jil} = \eta_{ij} f_{ijk} \eta_{ij}^*
	\]
	since both $f_{jil},f_{ijk}$ lie on the intersection of two horospheres which respectively touch $\partial \mathbb{H}^{3}$ at $\tilde{z}_i,\tilde{z}_j$. Indeed $\eta$ is a \emph{multiplicative 1-form} satisfying $\eta_{ji} = \eta_{ij}^{-1}$.
	
	We claim that there exists $A:F \to SL(2,\mathbb{C})$ such that
	\begin{align}\label{eq:etaexact}
			\eta_{ij}=A_{jil}A_{ijk}^{-1}.
	\end{align}
	To see this, consider a primal vertex $i$ and label its neighboring primal vertices as $v_0,v_1,v_2,\dots,v_s=v_0$. We denote by
	\[
		\prod_{j=1}^s \eta_{ij}:= \eta_{is} \cdots \eta_{i2} \eta_{i1}
	\]
	Then we have 
	\begin{gather*}
	(\prod_{j=1}^s \eta_{ij}) f_{i01} (\prod_{j=1}^s \eta_{ij})^* = f_{i01} \\
	(\prod_{j=1}^s \eta_{ij}) \left( \begin{array}{c}
	\tilde{z}_i \\ 1 \end{array} \right)  = (\prod_{j=1}^s \lambda_{ij}) \left( \begin{array}{c}
	\tilde{z}_i \\ 1 \end{array} \right) = \left( \begin{array}{c}
	\tilde{z}_i \\ 1 \end{array} \right) 
	\end{gather*}
	where the second equation holds because $f$ has constant integrated mean curvature-1. The first equation implies $\prod_{j=1}^s \eta_{ij}$ is conjugate to an element in $SU(2)$ and the second implies it has eigenvalues $1$. Thus $\prod_{j=1}^s \eta_{ij}$ is the identity and $\eta$ is a closed 1-form. Since the surface is assumed to be simply connected, by integration we obtain a mapping $A:F \to SL(2,\mathbb{C})$ satisfying Equation \eqref{eq:etaexact}. It is unique up to multiplication from the right by a constant matrix in $SL(2,C)$. Such constant matrix is determined so that the matrix $A$ furthermore satisfies
	\[
	f = A A^{*}
	\]
	
	Proposition \ref{prop:transit} implies $A^{-1}$ is indeed an osculating M\"{o}bius transformation by construction. There is another circle pattern $z:V \to \mathbb{C}$ and $A$ is the osculating M\"{o}bius transformation from $z$ to $\tilde{z}$. The cross ratios $X, \tilde{X}$ satisfy 
	\[
	\frac{|X|}{|\tilde{X}|} = |\lambda^2| = 1 
	\] 
	and hence the two circle patterns share the same modulus of cross ratios. On the other hand, the second condition in Definition \ref{def:discretecmc} implies that
	\[
	    0\leq \Arg X_{ij} = \ell_{ij}  \tan \frac{\alpha_{ij}}{2} + \Arg \tilde{X}_{ij} < \pi
	\] 
	and thus $z$ is Delaunay.
\end{proof}

The Weierstrass-type representation enables us to deduce some properties of discrete CMC-1 surfaces as in their smooth counterparts. For example, it is known that every smooth CMC-1 surface $f$ admits a dual CMC-1 surface $\tilde{f}$ such that their Hopf differentials are related by $\tilde{Q}=-Q$.

\begin{corollary}\label{cor:dualdisc}
	Suppose $f:V^{*} \to \mathbb{H}^3$ is a discrete CMC-1 surface with edge lengths $\ell$ and dihedral angle $\alpha$. Then there is a dual discrete CMC-1 surface $\tilde{f}:V^{*} \to \mathbb{H}^3$ such that
	\begin{equation}\label{eq:dual}
			\tilde{\ell} \tan \frac{\tilde{\alpha}}{2} = -\ell \tan \frac{\alpha}{2}.
	\end{equation}
\end{corollary}
\begin{proof}
	By Theorem \ref{thm:horo}, we know $f=AA^*$ for some osculating M\"{o}bius transformation $A$ from some circle pattern $z$ to $\tilde{z}$. Then $A^{-1}$ is the osculating M\"{o}bius transformation  from $\tilde{z}$ to $z$. It defines another discrete CMC-1 surface $\tilde{f}:= A^{-1} (A^{-1})^*$ which satisfies Equation \eqref{eq:dual}.
\end{proof}

\begin{remark}
	Theorem \ref{thm:horo} also holds for general circle patterns sharing the same modulus of cross ratios (not necessary Delaunay) if we broaden the definition of discrete CMC-1 surfaces as follows:
	\begin{enumerate}[(i)]
		\item In Definition \ref{def:horo} condition (2), edges are allowed to be realized as the longer circular arcs in the intersection of neighboring horospheres 
		\item In Definition \ref{def:horo} condition (3), the hyperbolic Gauss map is allowed to be non-Delaunay.
		\item Condition (2) of Definition \ref{def:discretecmc} is removed.
	\end{enumerate} 
\end{remark}

\section{A 1-parameter family of cross ratio systems} \label{sec:toda}

Hertrich-Jeromin \cite{Hertrich-Jeromin2000} introduced discrete CMC-1 surfaces in the context of integrable systems. These surfaces are constructed from special meshes in the plane that admit a real-valued solution to a discrete Toda-type equation \cite{Bobenko2002}. We show that such a solution induces a family of circle patterns sharing the same modulus of cross ratios and thus our discrete CMC-1 surfaces. In this way, our discrete CMC-1 surfaces generalize those from discrete integrable systems. 

In this section, we consider a cell decomposition $M=(V,E,F)$ of a simply connected domain that is not necessarily triangulated. For simplicity, we assume the domain is without boundary.

The following is a reformulation of the discrete Toda-type equation in \cite{Bobenko2002,Adler2001}. See the appendix in \cite{Gekhtman2016} for its relation to the Toda lattice.

\begin{definition}\label{def:Toda}
		Suppose $z:V(M) \to \mathbb{C}$ is a realization of  a cell decomposition $M$. A function $q:E(M) \to \mathbb{C}$ is a solution to the discrete Toda-type equation if it satisfies  
	\begin{align}
	\sum_j q_{ij} &=0 \quad \forall i \in V \label{eq:versum}\\
	\sum_{ij \in \phi} q_{ij} &=0 \quad \forall \phi \in F \label{eq:facesum}\\
	\sum_j \frac{q_{ij}}{z_j-z_i} &=0 \quad \forall i \in V \label{eq:verzsum}
	\end{align}
    The zero function $q\equiv 0$ is a trivial solution.
\end{definition}

We shall be interested in real-valued solutions to the discrete Toda-type equation. Given a generic realization, the only real-valued solution is the trivial solution. Heuristically, consider a triangulated disk with boundary. Equations \eqref{eq:versum} and \eqref{eq:verzsum} are imposed at interior vertices while Equation \eqref{eq:facesum} is imposed on faces. Assuming the constraints are linearly independent, with the Euler characteristic, the space of the real-valued solutions is $2|V_b|-|F|-3$, where $|V_b|$ is the number of boundary vertices. Thus, a generic realization of a triangulated disk with $2|V_b|-|F|-3<0$ does not admit non-trivial real-valued solutions.

\begin{example}
	The standard square grid $\mathbb{Z}^2$ admits a nontrivial solution $q$, where it takes values $+1$ on horizontal edges and $-1$ on vertical edges. Generally, every orthogonal circle pattern admits such a nontrivial $q$ \cite{Bobenko2002,Lam2018} (See Figure \ref{fig:ortho}).
	\begin{figure}[h!]
		\centering
		\includegraphics[width=0.6 \textwidth]{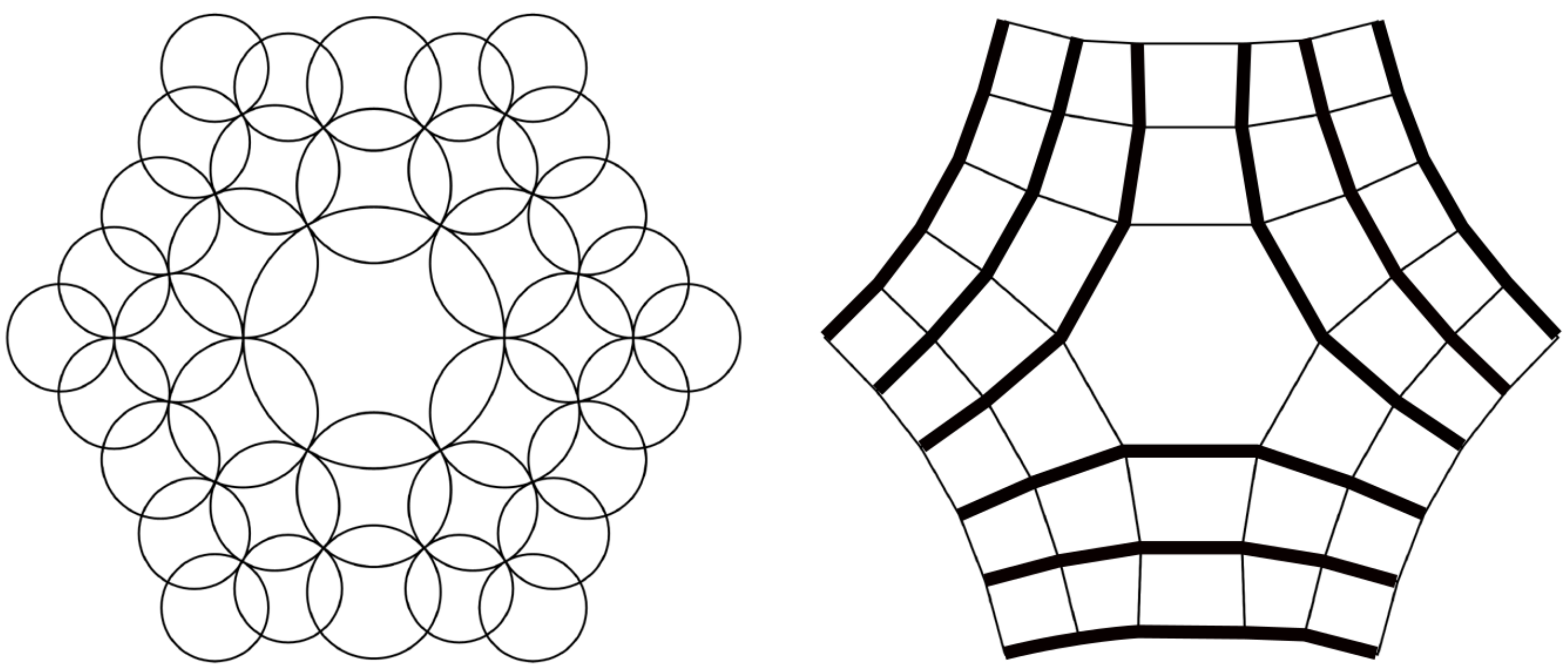}
		\caption{The intersection of an orthogonal circle pattern (left) yields a realization $z$ (right) with a nontrivial solution $q$ to the discrete Toda-type equation, where $q=1$ on thin edges and $q=-1$ on thick edges.}
		\label{fig:ortho}
	\end{figure}
\end{example}

Equations \eqref{eq:versum} and \eqref{eq:facesum} are related to a combinatorial object called a labeling on zig-zac paths of $M$. Given a cell decomposition $M$ and its dual $M^*$, we build the double $DM$ of $M$, which is a quadrilateral mesh as follows: The vertex set is $V(DM) = V(M) \cup V(M^*)$. A primal vertex $v \in V(M)$ and a dual vertex $f \in V(M^*)$ are joined by an edge in $DM$ if the vertex $v$ belongs to the face corresponding to the
dual vertex $f$. It forms a quadrilateral mesh, where each edge of $M$ corresponds to a face of $DM$. Notice that the graph $DM$ is bipartite: one can color the vertices of $M$ black and those of $M^*$ white. In this way, every edge of $DM$ connects two different color. See Figure \ref{fig:labeling}.

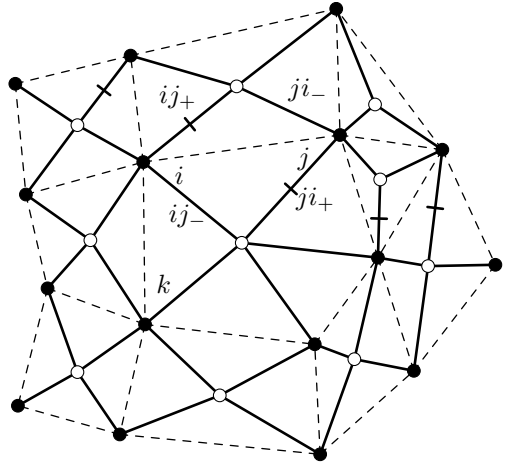
\begin{figure}
\definecolor{ffffff}{rgb}{1.,1.,1.}
\centering
\begin{tikzpicture}[line cap=round,line join=round,>=triangle 45,x=1.0cm,y=1.0cm,scale=1.2]
\clip(-4.42156497371259,-0.7594401433239386) rectangle (3.7636824637061514,4.473644004048038);
\draw [line width=0.5pt,dash pattern=on 3pt off 3pt] (-1.7,2.7)-- (-1.68,0.9);
\draw [line width=0.5pt,dash pattern=on 3pt off 3pt] (-1.68,0.9)-- (0.2,0.68);
\draw [line width=0.5pt,dash pattern=on 3pt off 3pt] (0.2,0.68)-- (0.9,1.64);
\draw [line width=0.5pt,dash pattern=on 3pt off 3pt] (0.9,1.64)-- (0.48,3.);
\draw [line width=0.5pt,dash pattern=on 3pt off 3pt] (0.48,3.)-- (-1.7,2.7);
\draw [line width=0.5pt,dash pattern=on 3pt off 3pt] (-1.68,0.9)-- (-1.96,-0.32);
\draw [line width=0.5pt,dash pattern=on 3pt off 3pt] (-1.96,-0.32)-- (0.26,-0.54);
\draw [line width=0.5pt,dash pattern=on 3pt off 3pt] (0.26,-0.54)-- (0.2,0.68);
\draw [line width=0.5pt,dash pattern=on 3pt off 3pt] (-3.,2.34)-- (-1.7,2.7);
\draw [line width=0.5pt,dash pattern=on 3pt off 3pt] (0.48,3.)-- (1.6134423312844166,2.836705187303399);
\draw [line width=0.5pt,dash pattern=on 3pt off 3pt] (1.6134423312844166,2.836705187303399)-- (0.9,1.64);
\draw [line width=0.5pt,dash pattern=on 3pt off 3pt] (-1.7,2.7)-- (-1.84,3.88);
\draw [line width=0.5pt,dash pattern=on 3pt off 3pt] (-1.84,3.88)-- (0.4379777255605342,4.3992130168631975);
\draw [line width=0.5pt,dash pattern=on 3pt off 3pt] (0.4379777255605342,4.3992130168631975)-- (0.48,3.);
\draw [line width=0.5pt,dash pattern=on 3pt off 3pt] (-3.,2.34)-- (-2.76,1.3);
\draw [line width=0.5pt,dash pattern=on 3pt off 3pt] (-2.76,1.3)-- (-1.68,0.9);
\draw [line width=0.5pt,dash pattern=on 3pt off 3pt] (-3.117085960043403,3.567786832326791)-- (-3.,2.34);
\draw [line width=0.5pt,dash pattern=on 3pt off 3pt] (-3.117085960043403,3.567786832326791)-- (-1.84,3.88);
\draw [line width=0.5pt,dash pattern=on 3pt off 3pt] (-2.76,1.3)-- (-3.088416091611113,-0.0016117874932995593);
\draw [line width=0.5pt,dash pattern=on 3pt off 3pt] (-1.96,-0.32)-- (-3.088416091611113,-0.0016117874932995593);
\draw [line width=0.5pt,dash pattern=on 3pt off 3pt] (0.26,-0.54)-- (1.2980737785292287,0.38543143634261384);
\draw [line width=0.5pt,dash pattern=on 3pt off 3pt] (1.2980737785292287,0.38543143634261384)-- (0.9,1.64);
\draw [line width=0.5pt,dash pattern=on 3pt off 3pt] (1.6134423312844166,2.836705187303399)-- (2.201174634146358,1.560896042066499);
\draw [line width=0.5pt,dash pattern=on 3pt off 3pt] (2.201174634146358,1.560896042066499)-- (1.2980737785292287,0.38543143634261384);
\draw [line width=0.5pt,dash pattern=on 3pt off 3pt] (1.6134423312844166,2.836705187303399)-- (0.4379777255605342,4.3992130168631975);
\draw [line width=1.pt] (-0.6084724722180448,1.804589923740963)-- (0.48,3.);
\draw [line width=1.pt] (-0.1278321998783764,2.460201831890974) -- (-6.402723396676448E-4,2.344388091849989);
\draw [line width=1.pt] (0.48,3.)-- (-0.6658122090826244,3.539116963894501);
\draw [line width=1.pt] (-0.6658122090826244,3.539116963894501)-- (-1.7,2.7);
\draw [line width=1.pt] (-1.1287142434820976,3.052768554575961) -- (-1.2370979656005268,3.1863484093185397);
\draw [line width=1.pt] (-1.7,2.7)-- (-0.6084724722180448,1.804589923740963);
\draw [line width=1.pt] (-0.6084724722180448,1.804589923740963)-- (-1.68,0.9);
\draw [line width=1.pt] (-1.68,0.9)-- (-2.285659775506999,1.833259792173253);
\draw [line width=1.pt] (-2.285659775506999,1.833259792173253)-- (-1.7,2.7);
\draw [line width=1.pt] (-1.7,2.7)-- (-2.429009117668448,3.1090689374101528);
\draw [line width=1.pt] (-2.429009117668448,3.1090689374101528)-- (-1.84,3.88);
\draw [line width=1.pt] (-2.202849474383677,3.5467515609437434) -- (-2.0661596432847715,3.4423173764664092);
\draw [line width=1.pt] (-1.84,3.88)-- (-0.6658122090826244,3.539116963894501);
\draw [line width=1.pt] (-3.117085960043403,3.567786832326791)-- (-2.429009117668448,3.1090689374101528);
\draw [line width=1.pt] (-2.429009117668448,3.1090689374101528)-- (-3.,2.34);
\draw [line width=1.pt] (-3.,2.34)-- (-2.285659775506999,1.833259792173253);
\draw [line width=1.pt] (-2.285659775506999,1.833259792173253)-- (-2.76,1.3);
\draw [line width=1.pt] (-2.76,1.3)-- (-2.429009117668448,0.3710965021264689);
\draw [line width=1.pt] (-2.429009117668448,0.3710965021264689)-- (-1.68,0.9);
\draw [line width=1.pt] (-3.088416091611113,-0.0016117874932995593)-- (-2.429009117668448,0.3710965021264689);
\draw [line width=1.pt] (-2.429009117668448,0.3710965021264689)-- (-1.96,-0.32);
\draw [line width=1.pt] (-1.96,-0.32)-- (-0.8521663538925082,0.11306768623585997);
\draw [line width=1.pt] (-0.8521663538925082,0.11306768623585997)-- (0.26,-0.54);
\draw [line width=1.pt] (0.26,-0.54)-- (0.638666804586562,0.5144458442879183);
\draw [line width=1.pt] (0.638666804586562,0.5144458442879183)-- (1.2980737785292287,0.38543143634261384);
\draw [line width=1.pt] (1.2980737785292287,0.38543143634261384)-- (1.4557580549068216,1.546561107850354);
\draw [line width=1.pt] (1.4557580549068216,1.546561107850354)-- (2.201174634146358,1.560896042066499);
\draw [line width=1.pt] (1.4557580549068216,1.546561107850354)-- (1.6134423312844166,2.836705187303399);
\draw [line width=1.pt] (1.4492258962089155,2.202067783863029) -- (1.6199744899823219,2.1811985112907237);
\draw [line width=1.pt] (0.9,1.64)-- (1.4557580549068216,1.546561107850354);
\draw [line width=1.pt] (0.48,3.)-- (0.9253654889094601,2.507001700332065);
\draw [line width=1.pt] (0.9253654889094601,2.507001700332065)-- (0.9,1.64);
\draw [line width=1.pt] (0.9986555635585633,2.0709855810625717) -- (0.8267099253508966,2.076016119269493);
\draw [line width=1.pt] (0.9253654889094601,2.507001700332065)-- (1.6134423312844166,2.836705187303399);
\draw [line width=1.pt] (0.4379777255605342,4.3992130168631975)-- (0.8680257520448805,3.3384278848684716);
\draw [line width=1.pt] (0.8680257520448805,3.3384278848684716)-- (0.48,3.);
\draw [line width=1.pt] (0.8680257520448805,3.3384278848684716)-- (1.6134423312844166,2.836705187303399);
\draw [line width=1.pt] (-0.6658122090826244,3.539116963894501)-- (0.4379777255605342,4.3992130168631975);
\draw [line width=1.pt] (-0.6084724722180448,1.804589923740963)-- (0.9,1.64);
\draw [line width=1.pt] (-0.6084724722180448,1.804589923740963)-- (0.2,0.68);
\draw [line width=1.pt] (-0.8521663538925082,0.11306768623585997)-- (-1.68,0.9);
\draw [line width=1.pt] (-0.8521663538925082,0.11306768623585997)-- (0.2,0.68);
\draw [line width=1.pt] (0.2,0.68)-- (0.638666804586562,0.5144458442879183);
\draw [line width=1.pt] (0.638666804586562,0.5144458442879183)-- (0.9,1.64);
\draw (-1.4542335909705943,2.7570497268164133) node[anchor=north west] {$i$};
\draw (-0.09241484043682818,2.9690689374101528) node[anchor=north west] {$j$};
\draw (-1.6262528015643332,3.65246630605595) node[anchor=north west] {$ij_+$};
\draw (-1.5192335909705943,2.328331831899775) node[anchor=north west] {$ij_-$};
\draw (-0.19275937994984252,3.732810845568965) node[anchor=north west] {$ji_-$};
\draw (-0.10975457730140782,2.52133663454821) node[anchor=north west] {$ji_+$};
\draw (-1.6562528015643332,1.5582357789310788) node[anchor=north west] {$k$};
\begin{scriptsize}
\draw [fill=black] (-1.7,2.7) circle (2.0pt);
\draw [fill=black] (-1.68,0.9) circle (2.0pt);
\draw [fill=black] (0.2,0.68) circle (2.0pt);
\draw [fill=black] (0.9,1.64) circle (2.0pt);
\draw [fill=black] (0.48,3.) circle (2.0pt);
\draw [fill=black] (1.6134423312844166,2.836705187303399) circle (2.0pt);
\draw [fill=black] (-3.,2.34) circle (2.0pt);
\draw [fill=black] (-1.96,-0.32) circle (2.0pt);
\draw [fill=black] (0.26,-0.54) circle (2.0pt);
\draw [fill=black] (-1.84,3.88) circle (2.0pt);
\draw [fill=black] (0.4379777255605342,4.3992130168631975) circle (2.0pt);
\draw [fill=black] (-2.76,1.3) circle (2.0pt);
\draw [fill=black] (-3.117085960043403,3.567786832326791) circle (2.0pt);
\draw [fill=black] (-3.088416091611113,-0.0016117874932995593) circle (2.0pt);
\draw [fill=black] (1.2980737785292287,0.38543143634261384) circle (2.0pt);
\draw [fill=black] (2.201174634146358,1.560896042066499) circle (2.0pt);
\draw [fill=ffffff] (-0.6084724722180448,1.804589923740963) circle (2.0pt);
\draw [fill=ffffff] (-0.6658122090826244,3.539116963894501) circle (2.0pt);
\draw [fill=ffffff] (-2.285659775506999,1.833259792173253) circle (2.0pt);
\draw [fill=ffffff] (-2.429009117668448,3.1090689374101528) circle (2.0pt);
\draw [fill=ffffff] (-2.429009117668448,0.3710965021264689) circle (2.0pt);
\draw [fill=ffffff] (-0.8521663538925082,0.11306768623585997) circle (2.0pt);
\draw [fill=ffffff] (0.638666804586562,0.5144458442879183) circle (2.0pt);
\draw [fill=ffffff] (1.4557580549068216,1.546561107850354) circle (2.0pt);
\draw [fill=ffffff] (0.9253654889094601,2.507001700332065) circle (2.0pt);
\draw [fill=ffffff] (0.8680257520448805,3.3384278848684716) circle (2.0pt);
\end{scriptsize}
\end{tikzpicture}
\caption{ The dotted lines and the black vertices form a cell decomposition $M$. The double $DM$ consists of the solid lines, the black and the white vertices. Every edge $\{ij\}$ of $M$ corresponds to a quadrilateral face of $DM$ enclosed by $\{ij_{+}\},\{ij_{-}\},\{ji_{+}\},\{ji_{-}\} \in E(DM)$. Under this notation, $\{ij_-\}$ and $\{ik_+\}$ represent the same edge of $E(DM)$. A collection of edges in $E(DM)$ are indicated where a labeling $\alpha$ has to take the same value.}
\label{fig:labeling}
\end{figure}

\begin{definition}
	A labeling is a function $\alpha: E(DM) \to \mathbb{C}$ on unoriented edges such that the values of $\alpha$ on two opposite edges in any quadrilateral face of $F(DM)$ are equal.
\end{definition}

Every labeling $\alpha$ on $E(DM)$ induces a function $q$ on $E(M)$ as follows. Let $\{ij\}\in E(M)$. It corresponds to a quadrilateral in $DM$. We denote by $\{ij_{+}\}, \{ij_{-}\} \in E(DM)$ the two edges in the quadrilateral sharing the vertex $i$ so that $\{ij_{+}\}$ is on the left of $\{ij\}$ and $\{ij_{-}\}$ is on the right. We define
\[
q_{ij} := \alpha_{ij_+}- \alpha_{ij_-} = \alpha_{ji_+}- \alpha_{ji_-} = q_{ji}
\]
since $\alpha_{ij_+} = \alpha_{ji_+}$ and $\alpha_{ij_-}= \alpha_{ji_-}$. By construction, it satisfies Equations \eqref{eq:versum} and \eqref{eq:facesum}. The following lemma is elementary.

\begin{lemma}
	A function $q:E(M) \to \mathbb{C}$ satisfies Equations \eqref{eq:versum} and \eqref{eq:facesum} 
	if and only if there exists a labeling $\alpha:E(DM) \to \mathbb{C}$ such that
	\[
	q_{ij} = \alpha_{ij_+} - \alpha_{ij_-}
	\] 
	The labeling is unique up to a constant. Furthermore, If $q$ is real-valued, $\alpha$ can be chosen to be real-valued as well.
\end{lemma}

We then relate Equations \eqref{eq:versum} and \eqref{eq:verzsum} to the infinitesimal change in cross ratios. Consider a family of cross ratio systems $X_t:E(M) \to \mathbb{C}$ on a triangle mesh and denote $X_0=X$, we consider its logarithmic derivative
	\[
q := \frac{d}{dt}(\log X_t )|_{t=0}
\]
By differentiating \eqref{eq:crproduct} \eqref{eq:crsum}, it is known \cite{Lam2019} that the function $q$ satisfies for every interior vertex $i$ with neighboring vertices $1,2,\dots n$
	\begin{align*}
0&=\sum_j q_{ij}  \\
0 &= q_{i1} X_{i1} + (q_{i1} + q_{i2}) X_{i1} X_{i2} +  \dotsc +  (q_{i1}+q_{i2} + \dotsc + q_{in}) X_{i1}X_{i2}\dotsm X_{in}
\end{align*}
Such a linear system defines a tangent space of the space of cross ratios at $X$. Substitute the developing $z$ for the cross ratios $X$, the second equation becomes
\[
\sum_j \frac{q_{ij}}{z_j-z_i} =0.
\]
\begin{proposition}[\cite{Lam2015a,Lam2019}]
	On a triangular mesh, a function $q:E \to \mathbb{C}$ describes an infinitesimal change in the cross ratio if and only if it satisfies Equations \eqref{eq:versum} and \eqref{eq:verzsum}.
\end{proposition}

To summarize, a solution to the discrete Toda-type equation represents a tangent vector in the space of cross ratios that is also induced from a labeling on the graph. In the following, we show that such a solution induces a canonical family of cross ratios.

\begin{theorem}\label{thm:inttoda}
	Suppose a realization $z:V(M) \to \mathbb{C}$ of a cell decomposition $M$ admits a solution $q:E(M) \to \mathbb{C}$ to the discrete Toda-type equation with labeling $\alpha$. We denote by $TM$ a triangulation of $M$ and $X:E(TM)\to \mathbb{C}$ the cross ratios.
	Then for small $t$, there is a 1-parameter family of cross ratios $X_{t}:E(TM)\to \mathbb{C}$
\begin{align*}
	X_{t,ij}:= \begin{cases}
	\frac{1- t \alpha_{ij_{-}}}{1 - t \alpha_{ij_{+}}} X_{ij} \quad &\text{ if } \{ij\} \in E(M)\\
	X_{ij}  \quad &\text{ if } \{ij\} \in E(TM)-E(M) 
	\end{cases}
\end{align*}
satisfying \eqref{eq:crproduct} and \eqref{eq:crsum}. It defines realizations $z_{t}:V(M) \to \mathbb{C}$ of $M$  and osculating M\"{o}bius transformations $\tilde{A}_{t}:F(M) \to SL(2,\mathbb{C})/\{\pm I\}$ from $z$ to $z_t$. Both $z_t$ and $\tilde{A}_t$ are independent of the triangulation chosen. 
\end{theorem}

\begin{proof}
	The function $X_t$ satisfies Equation \eqref{eq:crproduct} and \eqref{eq:crsum} by direct computation and the property of the labeling that if a vertex $i$ has neighboring vertices $1,2,\dots n$ in $M$, then $\alpha_{ik_{+}} =\alpha_{i(k+1)_{-}}$ for $k=1,2,3\dots,n$.
	
	For each $t$, there is a developing map $z_t$ and $\tilde{A}_t:F(TM)\to SL(2,\mathbb{C})/\{\pm I\}$ mapping $z$ to $z_t$. Recall that if $\{ijk\},\{jil\} \in F(TM)$ are two neighbouring faces sharing an edge $\{ij\}$, then $\tilde{A}_{t,jil}^{-1} \tilde{A}_{t,ijk}$ has eigenvectors $(z_i,1)^T$ and $(z_j,1)^T$ with eigenvalues  $\pm (\frac{X_{t,ij}}{X_{ij}})^{\frac{1}{2}}$ and  $\pm (\frac{X_{t,ij}}{X_{ij}})^{-\frac{1}{2}}$. If $\{ij\} \in E(TM)-E(M)$, we have $\tilde{A}_{t,jil}= \tilde{A}_{t,ijk}$. If $\{ij\} \in E(M)$, we have $\lambda_{ij}^2 = 	\frac{1- t \alpha_{ij_{-}}}{1 - t \alpha_{ij_{+}}} $. Hence $\tilde{A}_{t}$ is well defined on the faces of $M$ and is independent of the triangulation. Since $z_{t}$ is the image of $z$ under M\"{o}bius transformations $\tilde{A}_{t}$, we deduce that $z_{t}$ is independent of the triangulation as well.
\end{proof}

\begin{corollary}\label{cor:familyofx}
	If $q:E(M) \to \mathbb{R}$ is a real-valued solution to the discrete Toda-type equation, then the labeling $\alpha$ can be chosen to take real values and for any small $t \in \mathbb{R}$ 
	\begin{align*}
		\Im \log X_t &= \Im \log X \\
		\Re \log X_{\mathbf{i}t} &= \Re \log  X_{-\mathbf{i}t}
	\end{align*}
	where $\mathbf{i}=\sqrt{-1}$. In particular, for small $t>0$, the osculating M\"{o}bius transformation from $z_{\mathbf{i}t}$ to  $z_{-\mathbf{i}t}$ induces a discrete CMC-1 surface. 	
\end{corollary}
The above statement holds analogously if $q$ is purely imaginary or generally if $\Arg q \mod \pi$ is constant. 

Finally, we sketch how our discrete CMC-1 surfaces are related to those of Hertrich-Jeromin \cite{Hertrich-Jeromin2000}. In \cite{Hertrich-Jeromin2000}, Hertrich-Jeromin started with a quadrilateral mesh  $z:V(\mathbb{Z}^2) \to \mathbb{C}$ with factorized cross ratios. The vertices of $\mathbb{Z}^2$ can be colored black and white in such a way that each edge connects a black vertex to a white vertex. The black vertices form a sub-lattice $\mathbb{Z}^2_b$. The white vertices form another sub-lattice $\mathbb{Z}^2_w$, which naturally is a dual graph of $\mathbb{Z}^2_b$. The restriction $z|_{\mathbb{Z}^2_b}$ is known to possess a real-valued solution to the discrete Toda-type equation \cite{Bobenko2002}. By Theorem \ref{thm:inttoda}, it produces a family of osculating M\"{o}bius transformations $A_t$ on the dual graph, which turns out to be the ``Calapso transformations" \cite{Hertrich-Jeromin2000} restricted to the white vertices $\mathbb{Z}^2_w$. By considering the discrete CMC-1 surfaces in \cite{Hertrich-Jeromin2000} restricted to $\mathbb{Z}^2_b$, one obtains discrete CMC-1 surfaces considered in this article.

\section{Convergence to smooth surfaces} \label{sec:convergence}
In this section, we show that every simply connected umbilic-free CMC-1 surface can be approximated by our discrete CMC-1 surfaces. Based on the convergence result on circle patterns by He-Schramm \cite{Schramm1998} and B\"{u}cking \cite{Bucking2016}, we deduce that discrete osculating M\"{o}bius transformations converge to their smooth counterparts and so do discrete CMC-1 surfaces as a result of the Weierstrass-type representation.

We follow the notations in \cite{Bucking2018}. We consider a regular triangular lattice $T$ in the plane with acute angles $\alpha, \beta,\gamma \in (0,\pi/2)$ and edge directions
\[
\omega_1 = 1 = -\omega_4, \quad \omega_2 = e^{\mathbf{i} \beta} = - \omega_5, \quad \omega_3 = e^{\mathbf{i}(\alpha + \beta)} = -\omega_6
\]
as well as edge lengths
\[
L_1 = \sin \alpha = L_4, \quad L_2= \sin \gamma = L_5, \quad L_3= \sin \beta = L_6
\]
(See Figure \ref{fig:lattice}). Denote $ T^{(\epsilon)}$ the lattice $T$ scaled by $\epsilon >0$. Its vertex set is parameterized by the vertex position in complex coordinates
\[
V^{(\epsilon)} = \{ n \, \epsilon  \sin \alpha +m \,\epsilon \,  e^{\mathbf{i} \beta} \sin \gamma | m,n \in \mathbb{Z}\}.
\]
By abuse of notation,  $ T^{(\epsilon)}$ might refer to the geometric realization or solely the combinatorics. Its meaning shall be clear from the context. For a subcomplex of $ T^{(\epsilon)}$, its support is the union of vertices, edges and faces as a subset of $\mathbb{C}$.
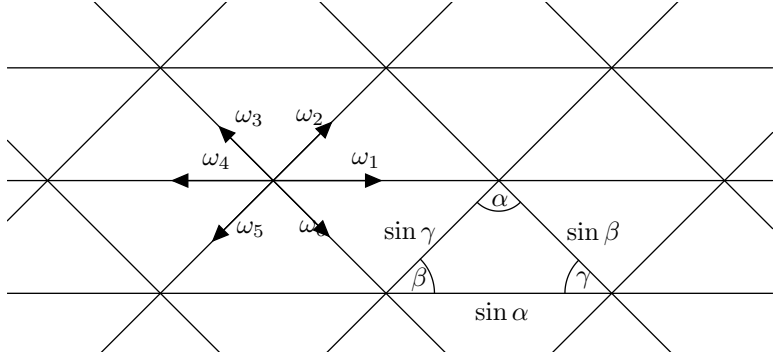
\begin{figure}
	\begin{tikzpicture}[line cap=round,line join=round,>=triangle 45,x=1.5cm,y=1.5cm]
	\clip(-2.3566636497721896,0.4849495367372822) rectangle (4.426543031054168,3.583209347655252);
	\draw [line width=.5pt,domain=-2.3566636497721896:4.426543031054168] plot(\x,{(-6.-3.*\x)/3.});
	\draw [line width=.5pt,domain=-2.3566636497721896:4.426543031054168] plot(\x,{(-0.-4.*\x)/4.});
	\draw [line width=.5pt,domain=-2.3566636497721896:4.426543031054168] plot(\x,{(--8.-4.*\x)/4.});
	\draw [line width=.5pt,domain=-2.3566636497721896:4.426543031054168] plot(\x,{(--16.-4.*\x)/4.});
	\draw [line width=.5pt,domain=-2.3566636497721896:4.426543031054168] plot(\x,{(--18.-3.*\x)/3.});
	\draw [line width=.5pt,domain=-2.3566636497721896:4.426543031054168] plot(\x,{(--12.--3.*\x)/3.});
	\draw [line width=.5pt,domain=-2.3566636497721896:4.426543031054168] plot(\x,{(--8.--4.*\x)/4.});
	\draw [line width=.5pt,domain=-2.3566636497721896:4.426543031054168] plot(\x,{(-0.-4.*\x)/-4.});
	\draw [line width=.5pt,domain=-2.3566636497721896:4.426543031054168] plot(\x,{(--8.-4.*\x)/-4.});
	\draw [line width=.5pt,domain=-2.3566636497721896:4.426543031054168] plot(\x,{(--12.-3.*\x)/-3.});
	\draw [line width=.5pt,domain=-2.3566636497721896:4.426543031054168] plot(\x,{(--8.-0.*\x)/8.});
	\draw [line width=.5pt,domain=-2.3566636497721896:4.426543031054168] plot(\x,{(--12.-0.*\x)/6.});
	\draw [line width=.5pt,domain=-2.3566636497721896:4.426543031054168] plot(\x,{(--12.-0.*\x)/4.});
	\draw [line width=.5pt,domain=-2.3566636497721896:4.426543031054168] plot(\x,{(--8.-0.*\x)/2.});
	\draw [line width=.5pt,domain=-2.3566636497721896:4.426543031054168] plot(\x,{(-0.-0.*\x)/6.});
	\draw [line width=.5pt,domain=-2.3566636497721896:4.426543031054168] plot(\x,{(-4.-0.*\x)/4.});
	\draw [line width=.5pt,domain=-2.3566636497721896:4.426543031054168] plot(\x,{(-4.-0.*\x)/2.});
	\draw [shift={(1.,1.)},line width=.5pt]  plot[domain=0.:0.7853981633974483,variable=\t]({1.*0.42908390214233005*cos(\t r)+0.*0.42908390214233005*sin(\t r)},{0.*0.42908390214233005*cos(\t r)+1.*0.42908390214233005*sin(\t r)});
	\draw [shift={(3.,1.)},line width=.5pt]  plot[domain=2.356194490192345:3.141592653589793,variable=\t]({1.*0.4147244327346166*cos(\t r)+0.*0.4147244327346166*sin(\t r)},{0.*0.4147244327346166*cos(\t r)+1.*0.4147244327346166*sin(\t r)});
	\draw [shift={(1.9989501719213416,2.000196688466617)},line width=.5pt]  plot[domain=3.9302350589057937:5.499692984893866,variable=\t]({1.*0.2716880770082122*cos(\t r)+0.*0.2716880770082122*sin(\t r)},{0.*0.2716880770082122*cos(\t r)+1.*0.2716880770082122*sin(\t r)});
	\draw (1.8519241955499378,1.9575669417242752) node[anchor=north west] {$\alpha$};
	\draw (1.1402743698636048,1.31376358047393) node[anchor=north west] {$\beta$};
	\draw (2.6025740212362705,1.2758659375166046) node[anchor=north west] {$\gamma$};
	\draw [->,line width=.5pt] (0.,2.) -- (0.9687284360176696,2.);
	\draw [->,line width=.5pt] (0.,2.) -- (0.5256116741926871,2.525611674192687);
	\draw [->,line width=.5pt] (0.,2.) -- (-0.4869891799370858,2.486989179937086);
	\draw [->,line width=.5pt] (0.,2.) -- (-0.9069545702288255,2.);
	\draw [->,line width=.5pt] (0.,2.) -- (-0.5454194102606018,1.4545805897393982);
	\draw [->,line width=.5pt] (0.,2.) -- (0.504184323657844,1.495815676342156);
	\draw (0.6143450811438403,2.333929229041483) node[anchor=north west] {$\omega_1$};
	\draw (0.12090400721070334,2.7097797311453187) node[anchor=north west] {$\omega_2$};
	\draw (-0.4168717459450493,2.700830909666656) node[anchor=north west] {$\omega_3$};
	\draw (-0.7100804976982452,2.3249804075628204) node[anchor=north west] {$\omega_4$};
	\draw (-0.40582056742371203,1.7164605470137537) node[anchor=north west] {$\omega_5$};
	\draw (0.14900636425337788,1.7254093684924163) node[anchor=north west] {$\omega_6$};
	\draw (1.7066406948486593,1.0000430435073614) node[anchor=north west] {$\sin{\alpha}$};
	\draw (2.519392877577667,1.7212046544070674) node[anchor=north west] {$\sin{\beta}$};
	\draw (0.90005814409916285,1.7212046544070674) node[anchor=north west] {$\sin{\gamma}$};
	\end{tikzpicture}
	\caption{Triangular lattice}
	\label{fig:lattice}
\end{figure}

\begin{proposition}[\cite{Bucking2016}]  \label{prop:ulrike}
	Let $h:\Omega \to \mathbb{C}$ be a locally univalent function and $K \subset \Omega$ be a compact set which is the closure of its simply connected interior domain $\Omega_K := \mbox{int}(K)$. Consider a triangular lattice $T$ with strictly acute angles. For each $\epsilon>0$, let $T^{(\epsilon)}_K$ be a maximal subcomplex of $ T^{(\epsilon)}$ whose support is contained in $K$ and is homeomorphic to a closed disk. 
	
	Then if $\epsilon >0$ is small enough (depending on $K$, $h$ and $T$), there exists another realization $h^{(\epsilon)}:V^{(\epsilon)}_K \to \mathbb{C}$ with the same combinatorics of $T^{(\epsilon)}_K$ such that 
	\begin{enumerate}
		\item $h^{(\epsilon)}$ is Delaunay and share the same modulus of cross ratios with the lattice $T^{(\epsilon)}_K$
		\item For all $z$ in the support of $T^{(\epsilon)}_K$, we have $|h^{(\epsilon)}_{PL}(w) - h(w)| \leq C \epsilon$ where $h^{(\epsilon)}_{PL}$ is the piecewise linear extension of $h^{(\epsilon)}$ over triangular faces and $C$ is a constant depending only on $K,f,T$.
	\end{enumerate}
\end{proposition}
The uniqueness of $h^{(\epsilon)}$ holds if \textit{scale factors} are prescribed at boundary vertices (See \cite{Bucking2016}). It was later improved to be $C^{\infty}$-convergence by B\"{u}cking \cite{Bucking2018} in the sense of locally uniform convergence as considered by He-Schramm \cite{Schramm1998}.

\begin{definition}Let $f: \Omega \to \mathbb{C}^{d}$. For each $\epsilon > 0$, let $f^{(\epsilon)}$ be defined on some subset $V^{(\epsilon)}_0 \subset V^{(\epsilon)}$ with values in $\mathbb{C}^d$. Assume for each $z \in \Omega$ there are some $\delta_1, \delta_2 >0$ such that $\{ v \in V^{(\epsilon)}: |v-z| < \delta_2 \} \subset V^{(\epsilon)}_0$ whenever $\epsilon \in (0, \delta_1)$. 
	
	Then we say that $f^{(\epsilon)}$ converges to $f$ locally uniformly in $\Omega$ if for every $\sigma>0$ and every $z \in \Omega$ there are $\delta_1, \delta_2 > 0$ such that $|f(z) - f^{(\epsilon)}(v)| < \sigma$ for every $\epsilon \in (0, \delta_1)$ and every $v  \in V^{(\epsilon)}$ with $|v-z| < \delta_2$.
\end{definition}

For $k=1,\dots,6$, we denote by $\tau^{(\epsilon)}_k : V^{(\epsilon)} \to V^{(\epsilon)}$ the map that combinatorially shifts vertices in the $k$-th direction \[\tau^{(\epsilon)}_k(v) = v + \epsilon L_k \omega_k.\] For any subset $W \subset V^{\epsilon}$, a vertex $v \in W$ is an interior vertex of $W$ if all the six neighbouring vertices are in $W$. Write $W_{1}$ the set of interior vertices of $W$ and inductively $W_{r}$ the set of interior vertices of $W_{r-1}$.

Given a function $\eta : W \to \mathbb{C}$, the discrete directional derivative $\partial_k^{\epsilon} \eta : W_1 \to \mathbb{C}$ is defined as
\[
\partial_k^{\epsilon} \eta (v) = \frac{1}{\epsilon L_k} ( \eta( v + \epsilon L_k \omega_k) - \eta(v)).
\]
It is analogous to a directional derivative of a differentiable function $f: \Omega \to \mathbb{C}$, where we write 
\[
\partial_k f(z) := \lim_{t\to0} \frac{f(z+ t \omega_k) - f(z)}{t}  \quad \text{ for } k=1,2,\dots,6.
\]

\begin{definition}\label{def:cinfcon} 
	Let $n \in \mathbb{N}$ and suppose that $f$ is $C^{n}$-smooth. Then we say $f^{(\epsilon)}$ converges to $f$ in $C^{n}(\Omega)$ if for every sequence $k_1,\dots,k_j \in \{1,\dots,6\}$ with $j \leq n$ the functions $\partial_{k_j}\partial_{k_{j-1}}\dots\partial_{k_1} f^{(\epsilon)}$ converges to $\partial_{k_j}\partial_{k_{j-1}}\dots\partial_{k_1} f$ locally uniformly in $\Omega$. If this holds for all $n \in \mathbb{N}$, the convergence is said to be in $C^{\infty}(\Omega)$.
\end{definition}

It is known that  $C^{\infty}(\Omega)$-convergence has some nice properties. 

\begin{lemma}[\cite{Schramm1998}] \label{lem:hesch} Suppose that $f^{(\epsilon)}, g^{(\epsilon)}, h^{(\epsilon)}$ converge in $C^{\infty}(\Omega)$ to functions $f,g,h: \Omega \to \mathbb{C}$ and suppose that $h\neq 0$ in $\Omega$. Then the following convergences are in $C^{\infty}(\Omega)$
	\begin{enumerate}
		\item $f^{(\epsilon)} + g^{(\epsilon)} \to f +g$
		\item $f^{(\epsilon)} g^{(\epsilon)} \to f g$
		\item $1/h^{(\epsilon)} \to 1/h$
		\item if $h^{(\epsilon)}>0$, then $\sqrt{h^{(\epsilon)}} \to \sqrt{h}$
		\item $|h^{(\epsilon)}| \to |h|$
	\end{enumerate}
\end{lemma}

We consider cross ratios associated to edges. Denote by $X^{(\epsilon)}$ the cross ratios of the lattice $ T^{(\epsilon)}$ while denote by $\tilde{X}^{(\epsilon)}$ the cross ratios of $h^{(\epsilon)}$ in Proposition \ref{prop:ulrike}. Notice since $ T^{(\epsilon)}$ differs from $ T$ by scaling, we have $X^{(\epsilon)}=X$ independent of $\epsilon$. On the other hand $\log (\tilde{X}^{(\epsilon)}/X)$ is purely imaginary since $|\tilde{X}^{(\epsilon)}|/|X| =1$. For $k=1,2,3$, we define $s^{(\epsilon)}_k: V^{(\epsilon)}_{K} \to \mathbb{R}$
\[
s^{(\epsilon)}_k(v) := 	\frac{1}{\mathbf{i} \epsilon^2} \log \frac{\tilde{X}^{(\epsilon)}(e)}{X(e)}
\] 
where $e$ is the edge joining $v$ and $\tau^{(\epsilon)}_k (v)$. It plays the role of a discretization of the Schwarzian derivative (See \cite{Bobenko2003}).

Following the result by He-Schramm \cite{Schramm1998}, B\"{u}cking introduced contact transformation $Z_k: V^{(\epsilon)} \to SL(2,\mathbb{C})$ such that $Z_{k}(v)$ is the M\"{o}bius transformation mapping $0,\epsilon \omega_k L_k, \epsilon \omega_{k+1} L_{k+1}$ to $h^{(\epsilon)}(v),h^{(\epsilon)}(\tau_k(v)),h^{(\epsilon)}(\tau_{k+1}(v))$. In \cite[Section 5]{Bucking2018}, she proved the convergence of the contact transformations.

\begin{proposition}[\cite{Bucking2018}] \label{prop:schconv}
	Under the notations in Proposition \ref{prop:ulrike}, the mappings $h^{(\epsilon)}$ converge in $C^{\infty}(\Omega_K)$ to $h$. Furthermore, the discrete Schwarzian derivative in the $k$-th direction
		\begin{gather*}
	\lim_{\epsilon \to 0} s^{(\epsilon)}_1 = \frac{L_1}{2} \Re( \omega_2 \omega_3 S_h) \\
	\lim_{\epsilon \to 0} s^{(\epsilon)}_2 =- \frac{L_2}{2} \Re( \omega_1 \omega_3 S_h)\\
	\lim_{\epsilon \to 0} s^{(\epsilon)}_3 =  \frac{L_3}{2} \Re( \omega_1 \omega_2 S_h)
	\end{gather*}
	where $S_h$ is the classical Schwarzian derivative of $h$. For each $k=1,2,\dots,6$, the contact transformations	$Z_k$ converge to a complex analytic function $\mathcal{Z}$ in $C^{\infty}(\Omega)$ satisfying
	\[
	d\mathcal{Z} = \mathcal{Z} \left( \begin{array}{cc}
	0 & 1 \\ -\frac{S_h}{2} & 0
	\end{array}  \right) dz
	\] 
\end{proposition}

As a special case, when $T$ is the regular triangular lattice consisting of equilateral triangles of unit lengths, then $s_k$ converges to $ - \Re(e^{\mathbf{i}\frac{2k \pi}{3}} S_h)/2$. The proposition immediately implies the convergence of discrete osculating M\"{o}bius transformations.

\begin{proposition}\label{prop:distosmoothosc} Under the notations in Proposition \ref{prop:ulrike}, discrete osculating M\"{o}bius transformations  $A^{(\epsilon)}$ from the regular lattice $T^{(\epsilon)}_K$ to the realization $h^{(\epsilon)}$ converge to the smooth osculating M\"{o}bius transformation $A_h$ in $C^{\infty}(\Omega_K)$.
\end{proposition}
\begin{proof}
Define $B^{(\epsilon)}:V^{(\epsilon)}_K \to SL(2,\mathbb{C})$ to be the translation matrix
\[
B^{(\epsilon)}(v) = \left( \begin{array}{cc}
1 & z_v \\ 0 & 1
\end{array} \right)
\]
where $z_v$ is the complex coordinate of $v$ in the lattice $T^{(\epsilon)}_{K}$. Then we have the following relation between contact transformations and osculating M\"{o}bius transformations
\begin{align*}
Z^{(\epsilon)}_0(v) = A^{(\epsilon)}_{v \tau_0(v) \tau_1(v)} B^{(\epsilon)}(v)
\end{align*}
Notice that $B^{(\epsilon)}$ converges in $C^{\infty}(\Omega)$ to $B(z) = \left( \begin{array}{cc}
1 & z \\ 0 & 1
\end{array} \right)$. Thus $A^{(\epsilon)}$ converges to $\tilde{A} = \mathcal{Z} B^{-1}$ and satisfies
\begin{align*} 
\tilde{A}^{-1} d\tilde{A} &= B \mathcal{Z}^{-1} d\mathcal{Z}  B^{-1} + dB B^{-1}  \\
&= -\frac{S_h(z)}{2}\left( \begin{array}{cc}
z & -z^2 \\ 1 & -z
\end{array}\right) dz \\ 
&= A_h^{-1} dA_h 
\end{align*}
It implies $\tilde{A} = C A_h$ for some constant $C \in SL(2, \mathbb{C})$. For every $z$ we further know that $\tilde{A}(z)$ and $A_h(z)$ are M\"{o}bius transformations mapping $z$ to $h(z)$. Hence $C= \pm I$ and $A^{(\epsilon)}$ converges to $\tilde{A}=A_h$.
\end{proof}

The convergence of osculating M\"{o}bius transformations implies the convergence of discrete CMC-1 surfaces.

\begin{theorem}\label{thm:converge}
	Suppose $f:\Omega \to \mathbb{H}^3$ is a conformal immersion of an umbilic-free CMC-1 surface. Let  $K \subset \Omega$ be a compact set which is the closure of its simply connected interior domain $\Omega_K := \mbox{Int}(K)$. Consider a triangular lattice $T$ with strictly acute angles. For each $\epsilon>0$, let $T^{(\epsilon)}_K$ be a maximal subcomplex of scaled lattice $T^{(\epsilon)}$ whose support is contained in $K$ and is homeomorphic to a closed disk. 
	
	Then for $\epsilon >0$ small enough (depending on $K,f,T$), there exists discrete CMC-1 surface $f^{(\epsilon)}: (V^{(\epsilon)}_K)^* \to  \mathbb{H}^3$ and we have $f^{(\epsilon)}$ converging to $f$ in $C^{\infty}(\Omega_K)$ as $\epsilon \to 0$.
\end{theorem}
\begin{proof}
	Proposition \ref{prop:smoothoscu} implies that there exists locally univalent functions $g,\tilde{g}:\Omega \to \mathbb{C}$ such that
	\[
	f = A A^*
	\]
	where $A=A_{\tilde{g}} A_g^{-1}$ is the osculating M\"{o}bius transformation from $g$ to $\tilde{g}$. For $\epsilon > 0$ small, there exists Delaunay circle patterns $g^{(\epsilon)}$ and $\tilde{g}^{(\epsilon)}$ defined on $T^{(\epsilon)}_{K}$ sharing the same modulus of cross ratios that converge to $g$ and $\tilde{g}$ respectively. They induce a discrete CMC-1 surface defined on the dual graph of $T^{(\epsilon)}_{K}$ via 
	\[
	f^{(\epsilon)} = A^{(\epsilon)} (A^{(\epsilon)})^*
	\]
	where $A^{(\epsilon)} = A_{\tilde{g}^{(\epsilon)}} A_{g^{(\epsilon)}}^{-1}$ is the discrete osculating M\"{o}bius transformation from   $g^{(\epsilon)}$ to $\tilde{g}^{(\epsilon)}$. By Proposition \ref{prop:distosmoothosc} and Lemma \ref{lem:hesch}, $A^{(\epsilon)}$ converges in $C^{\infty}(\Omega)$ to $A_{\tilde{g}} A_{g}^{-1}=A$, which implies $f^{(\epsilon)}$ converges to $f$.
\end{proof}

Combining with Proposition \ref{prop:schconv}, we obtain convergence to the Hopf differential as well.
\begin{corollary}
	Under the notation of Theorem \ref{thm:converge}, the Hopf differential $Q$ of $f$ satisfies for $k=1,2,3,\dots 6$
	\begin{align*}
	L_k \Re( \omega_{k+1} \omega_{k+2} Q) = \lim_{\epsilon \to 0} \frac{\ell_k^{(\epsilon)}}{\epsilon^2} \tan\frac{\alpha_k^{(\epsilon)}}{2}
	\end{align*}
	where $\ell_k,\alpha_k: V^{(\epsilon)}_{K} \to \mathbb{R}$ are the edge lengths and the dihedral angles in the $k$-th direction, i.e.
	\[
	\ell^{(\epsilon)}_k(v) := 	\ell^{(\epsilon)}_e, \quad \alpha^{(\epsilon)}_k(v) :=  \alpha^{(\epsilon)}_e
	\] 
	and $e$ is the edge joining $v$ and $\tau^{\epsilon}_k v$.
\end{corollary}

\section{Circle patterns with the same intersection angles} \label{sec:circarg}

Previous sections focus on circle patterns with the same modulus of cross ratios, i.e. $|X|= |\tilde{X}|$. In fact, similar results hold for circle patterns with the same intersection angles, i.e. $\Arg X = \Arg \tilde{X}$. The analogues are considered in this section. As a replacement of horospherical nets in $\mathbb{H}^3$, we consider equidistant nets. (See section \ref{sec:umbilic} for the definition of equidistants.)
\begin{definition}\label{def:equidistant}
	An equidistant net is a realization $f:V^{*} \to \mathbb{H}^{3}$ of a dual mesh such 
	that
	\begin{enumerate}
		\item  Each vertex is associated with an equidistant that contains the vertex and the neighboring vertices. 
		\item Every edge is realized as a circular arc in the intersection of the two neighboring equidistants. 
		\item For each face of $f$, the equidistants associated to the vertices of the face intersect at a common vertex on $\partial\mathbb{H}^3$. It defines a circle pattern $\tilde{z}:V \to \mathbb{C} \cup \{\infty\}$ regarded as the hyperbolic Gauss map.
	\end{enumerate} 
\end{definition}

In the case of quadrilateral meshes, equidistant nets resemble asymptotic parametrizations in Euclidean space \cite{Bobenko1999} while horospherical nets resemble curvature line parametrizations. Unlike horospherical nets, integrated mean curvature is not defined on equidistant nets. However the assumptions on equidistant nets are strong enough to deduce that they correspond to a pair of circle patterns sharing the same intersection angles. The following is an analogue of Theorem \ref{thm:horo}.
\begin{theorem}\label{thm:equid}
		Given two Delaunay circle patterns $z,\tilde{z}:V\to \mathbb{C}$ with cross ratios $X,\tilde{X}$ such that $\Arg X = \Arg \tilde{X}$. Let $A:V^{*} \to SL(2,\mathbb{C})$ be the osculating M\"{o}bius transformation from $z$ to $\tilde{z}$. Then the realization $f:V^{*} \to \mathbb{H}^3$ of the dual graph 
	given by
	\[
	f:= A A^{*}
	\]
	is an equidistant net with hyperbolic Gauss map $\tilde{z}$.

	Conversely, suppose $f:V^{*} \to \mathbb{H}^3$ is an equidistant net with hyperbolic Gauss map $\tilde{z}:V \to \mathbb{C}$. Then there exists a circle pattern $z$ such that $f=A A^{*}$ and $A$ is the osculating M\"{o}bius transformation from $z$ to $\tilde{z}$. The cross ratios $X,\tilde{X}$ of $z,\tilde{z}$ satisfy $\Arg X = \Arg \tilde{X}$.	
\end{theorem}
\begin{proof}
	Let $A:F \to SL(2,\mathbb{C})$ be the osculating M\"{o}bius transformations mapping $z$ to $\tilde{z}$. Pick a dual vertex $\{ijk\}\in F=V^{*}$. Then there exists a unique equidistant passing through $f_{ijk}$ and intersect $\partial \mathbb{H}^3$ at the circumcircle of  $\tilde{z}_i,\tilde{z}_j,\tilde{z}_k$. We claim that this equidistant contains the neighbouring vertices of $f_{ijk}$ as well. Let $\{jki\}$ be a neighbouring dual vertex. Since $X_{ij}/\tilde{X}_{ij}$ is real valued, the transition matrix $A_{jli}A^{-1}_{ijk}$ is a scaling with fixed points at $\tilde{z}_i$ and $\tilde{z}_j$. Hence $f_{jki}$ lies on the circular arc through $f_{ijk}$,  $\tilde{z}_i$ and $\tilde{z}_j$ (not necessary to be a geodesic). In particular $f_{jki}$ lies on the equidistant through $f_{ijk}, \tilde{z}_i, \tilde{z}_j$ and $\tilde{z}_k$.
	
	Conversely, given an equidistant net $f$ with Gauss map $z$. We define $\eta$ as in Equation \eqref{eq:eta} with real eigenvalues $\lambda$ such that $\eta_{ij}$ is a scaling sending $f_{ijk}$ to $f_{jil}$ with fixed points $\tilde{z}_i$ and $\tilde{z}_j$. Notice that $\eta_{ij}=\eta^{-1}_{ji}$. We claim that for every primal vertex $i$ with neighboring primal vertices denoted as $v_0,v_1,v_2,\dots,v_s=v_0$, we have
	\[
	\prod_{j=1}^s \eta_{ij} = I.
	\]
	To see this, observe that
		\begin{gather}
(\prod_{j=1}^s \eta_{ij}) \left( \begin{array}{c}
\tilde{z}_i \\ 1 \end{array} \right)  = (\prod_{j=1}^s \lambda_{ij}) \left( \begin{array}{c}
\tilde{z}_i \\ 1 \end{array} \right)  \label{eq:etaprod1} \\
	(\prod_{j=1}^s \eta_{ij}) f_{i01} (\prod_{j=1}^s \eta_{ij})^* = f_{i01}  \label{eq:etaprod2}
	\end{gather} 
	Equation \eqref{eq:etaprod2} implies $\prod_j \eta_{ij}$ is conjugate to some element in $SU(2,\mathbb{C})$. Equation $\eqref{eq:etaprod1}$ implies $\prod_j \eta_{ij}$ has positive real eigenvalue $(\prod_j \lambda)$. Thus we deduce that the eigenvalue is $1$. It implies that $\prod_j \eta_{ij}$ is the identity and hence there exists $A:F \to SL(2,\mathbb{C})$ such that \[
		\eta_{ij}=A_{jil}A_{ijk}^{-1}.
	\]
	Applying $A^{-1}$ to $\tilde{z}$, we obtain a new realization $z$ with cross ratios $\tilde{X}$ satisfying $X/\tilde{X} = \lambda$. Since $\lambda$ is positively real, we have $\Arg X = \Arg \tilde{X}$.
\end{proof}

Following Corollary \ref{cor:familyofx}, equidistant nets also arise from the discrete Toda-type equation. 

As an analogue of Theorem \ref{thm:converge}, smooth CMC-1 surfaces can be approximated by equidistant nets. The proof can be carried over analogously, by establishing the convergence of discrete osculating M\"{o}bius transformations. For circle patterns induced from circle packings, the convergence holds as a result of He-Schramm \cite{Schramm1998}. 

\section{Minimal surfaces in Euclidean space}\label{sec:minimal}

Smooth minimal surfaces in $\mathbb{R}^3$ share the same holomorphic data as CMC-1 surfaces in $\mathbb{H}^3$. However in the discrete theory,  the holomorphic data no longer remains the same. In the section, we discuss the relation to discrete minimal surfaces in $\mathbb{R}^3$. 

In \cite{Lam2017,Lam2018}, it was shown that every discrete minimal surface in $\mathbb{R}^3$ corresponds to an infinitesimal deformation of a circle pattern. It can be regarded as a pair of circle patterns that are infinitesimally close to each other. Equivalently, the holomorphic data involves cross ratio $X$ and its first-order change $\dot{X}$.  Previous construction is based on reciprocal parallel meshes from the rigidity theory \cite{Lam2015}. The approach developed in this article provides another perspective of the construction in terms of osculating M\"{o}bius vector fields.

A M\"{o}bius vector field is generated by an infinitesimal M\"{o}bius transformation. It is a quadratic vector field 
\[
(-\gamma z^2 +2\alpha z +\beta) \frac{\partial}{\partial z}
\] corresponding to an element \[\left(\begin{array}{cc} \alpha & \beta\\ \gamma &-\alpha
\end{array}\right) \in sl(2,\mathbb{C}).\] Analogous to osculating M\"{o}bius transformations, every holomorphic vector field $h(z) \frac{ \partial}{\partial z}$ on a domain $\Omega$ is associated with an \emph{osculating M\"{o}bius vector field} $a:\Omega \to sl(2,\mathbb{C})$ that coincides with the 2-jet of $h$ at every point in $\Omega$
\begin{align*}
	h(z) &= -\gamma z^2 +2\alpha z +\beta\\
		h'(z) &= \partial_w (-\gamma w^2 +2\alpha w +\beta) |_{w=z}\\
			h''(z) &=\partial^2_w (-\gamma w^2 +2\alpha w +\beta) |_{w=z}
\end{align*}
which gives
\[
a(z) = \left(
\begin{array}{cc}
\frac{1}{2} \left(h'(z)-z h''(z)\right) & \frac{1}{2} \left(z^2 h''(z)-2 z h'(z)+2
h(z)\right) \\
-\frac{1}{2} h''(z) & \frac{1}{2} \left(z h''(z)-h'(z)\right) \\
\end{array}
\right)
\]
and
\[
da = -\frac{h'''(z)}{2}  \left(
\begin{array}{cc}
z & -z^2 \\
1 & -z \\
\end{array}
\right) dz.
\]
The coefficient $h'''$ is the Schwarzian derivative of a holomorphic vector field. It vanishes identically if and only if the vector field is globally generated by a M\"{o}bius transformation. Furthermore, the osculating M\"{o}bius vector field defines a holomorphic null curve in $sl(2,\mathbb{C})$ in the sense that the Killing form evaluated on the tangent vectors vanishes. We denote by $\mathbf{i}:sl(2,\mathbb{C}) \to \mathbb{C}^3$ an isomorphism between $sl(2,\mathbb{C})$ equipped with the Killing form and $\mathbb{C}^3$ equipped with the standard complex bilinear form. Then the composition of mappings
\begin{align}\label{eq:minimalsur}
	\Re(\mathbf{i}\circ a):\Omega \to sl(2,\mathbb{C}) \to \mathbb{C}^3 \to  \mathbb{R}^3
\end{align}
yields a minimal surface in $\mathbb{R}^3$. One can verify that every minimal surface in $\mathbb{R}^3$ correspond to an osculating M\"{o}bius vector field. In fact, it is related to a well known correspondence between minimal surfaces in $\mathbb{R}^3$ and holomorphic null curves in $\mathbb{C}^3$ (See references in \cite{Small1994}.  

Given a circle pattern $z:V \to \mathbb{C}$ and a vector field $\dot{z}:V \to \mathbb{C}$ denoting an infinitesimal deformation of the vertices, it is associated with an osculating M\"{o}bius vector field $a:F \to sl(2,\mathbb{C})$ analogously: for every face $\{ijk\}$, $a_{ijk}$ is the unique quadratic vector field that coincides with $\dot{z}_i,\dot{z}_j,\dot{z}_k$ at vertices $z_i,z_j,z_k$. Equation \eqref{eq:minimalsur} induces a realization of the dual graph in $\mathbb{R}^3$. The characterization has been carried out in \cite{Lam2015a,Lam2018} when $\dot{z}$ preserves the modulus of cross ratios $|X|$ or the intersection angles $\Arg X$. Such realizations are regarded as discrete minimal surfaces.

We conclude with a comparison between discrete minimal surfaces in $\mathbb{R}^3$ and CMC-1 surfaces in $\mathbb{H}^3$ in Table \ref{tab:1} and \ref{tab:2}.

\begin{table}[h]
	\caption{Modulus of cross ratios $|X|$ is preserved} 	\label{tab:1}
	\begin{tabular}{ c| c| c } 
		& Minimal surfaces in $\mathbb{R}^3$ & CMC-1 surface in $\mathbb{H}^3$  \\
		Edges & Geodesics & Circular arcs \\
		Faces & Piecewise linear& Piecewise horospherical\\ 
		Holomorphic data & $X,\dot{X}$ s.t. $\Re(\log \dot{X})\equiv0$  & $X,\tilde{X}$ s.t. $\Re \log (\tilde{X}/X) \equiv 0$
	\end{tabular}
\end{table}

\begin{table}[h]
	\caption{Intersection angles $\Arg X$ is preserved} 	\label{tab:2}
	\begin{tabular}{ c| c| c } 
		& Minimal surfaces in $\mathbb{R}^3$ & CMC-1 surface in $\mathbb{H}^3$  \\
		Edges & Geodesics & Circular arcs \\
		Vertex stars & Planar & Equidistant \\ 
		Holomorphic data & $X,\dot{X}$ with $\Im(\dot{X})\equiv0$  & $X,\tilde{X}$ s.t. $\Im \log (\tilde{X}/X) \equiv 0$
	\end{tabular}
\end{table}

\section*{Acknowledgment}

The author would like to thank Jean-Marc Schlenker for fruitful discussions.
\bibliographystyle{crelle}
\bibliography{holomorphicquad}

\begin{thebibliography}{10}
\providecommand{\url}[1]{\texttt{#1}}
\providecommand{\urlprefix}{URL }
\expandafter\ifx\csname urlstyle\endcsname\relax
  \providecommand{\doi}[1]{doi:\discretionary{}{}{}#1}\else
  \providecommand{\doi}{doi:\discretionary{}{}{}\begingroup
  \urlstyle{rm}\Url}\fi

\bibitem{Adler2001}
\textit{V.~E. Adler}, Discrete equations on planar graphs, J. Phys. A
  \textbf{34} (2001), no.~48, 10453--10460.

\bibitem{Anderson1998}
\textit{C.~G. Anderson}, Projective structures on {R}iemann surfaces and
  developing maps to {H}(3) and {CP}(n), ProQuest LLC, Ann Arbor, MI, 1998.
  Thesis (Ph.D.)--University of California, Berkeley.

\bibitem{Bobenko1996}
\textit{A.~Bobenko} and \textit{U.~Pinkall}, Discrete isothermic surfaces, J.
  Reine Angew. Math. \textbf{475} (1996), 187--208.

\bibitem{Bobenko2003}
\textit{A.~I. Bobenko} and \textit{T.~Hoffmann}, Hexagonal circle patterns and
  integrable systems: patterns with constant angles, Duke Math. J. \textbf{116}
  (2003), no.~3, 525--566.

\bibitem{Bobenko2006}
\textit{A.~I. Bobenko}, \textit{T.~Hoffmann} and \textit{B.~A. Springborn},
  Minimal surfaces from circle patterns: geometry from combinatorics, Ann. of
  Math. (2) \textbf{164} (2006), no.~1, 231--264.

\bibitem{Bobenko1999}
\textit{A.~I. Bobenko} and \textit{U.~Pinkall}, Discretization of surfaces and
  integrable systems, in: Discrete integrable geometry and physics ({V}ienna,
  1996), Oxford Univ. Press, New York, 1999, \textit{Oxford Lecture Ser. Math.
  Appl.}, volume~16, 3--58.

\bibitem{Bobenko2010}
\textit{A.~I. Bobenko}, \textit{U.~Pinkall} and \textit{B.~A. Springborn},
  Discrete conformal maps and ideal hyperbolic polyhedra, Geom. Topol.
  \textbf{19} (2015), no.~4, 2155--2215.

\bibitem{Bobenko2010a}
\textit{A.~I. Bobenko}, \textit{H.~Pottmann} and \textit{J.~Wallner}, A
  curvature theory for discrete surfaces based on mesh parallelity, Math. Ann.
  \textbf{348} (2010), no.~1, 1--24.

\bibitem{Bobenko2002}
\textit{A.~I. Bobenko} and \textit{Y.~B. Suris}, Integrable systems on
  quad-graphs, Int. Math. Res. Not.  (2002), no.~11, 573--611.

\bibitem{Brock2019}
\textit{M.~Bridgeman}, \textit{J.~Brock} and \textit{K.~Bromberg}, Schwarzian
  derivatives, projective structures, and the {W}eil-{P}etersson gradient flow
  for renormalized volume, Duke Math. J. \textbf{168} (2019), no.~5, 867--896.

\bibitem{Bryant1987}
\textit{R.~L. Bryant}, Surfaces of mean curvature one in hyperbolic space, in:
  Th\'eorie des vari\'et\'es minimales et applications, Soci\'et\'e
  math\'ematique de France, 1987, number 154-155 in Ast\'erisque, 321--347.

\bibitem{Bucking2016}
\textit{U.~B\"{u}cking}, Approximation of conformal mappings using conformally
  equivalent triangular lattices, in: Advances in discrete differential
  geometry, Springer, [Berlin], 2016, 133--149.

\bibitem{Bucking2018}
\textit{U.~B\"{u}cking}, {$C^\infty$}-convergence of conformal mappings for
  conformally equivalent triangular lattices, Results Math. \textbf{73} (2018),
  no.~2, 84.

\bibitem{Epstein1984}
\textit{C.~Epstein}, Envelopes of horospheres and weingarten surfaces in
  hyperbolic 3-spaces  (1984). Preprint.
  \url{https://www.math.upenn.edu/~cle/papers/WeingartenSurfaces.pdf}.

\bibitem{Gekhtman2016}
\textit{M.~Gekhtman}, \textit{M.~Shapiro}, \textit{S.~Tabachnikov} and
  \textit{A.~Vainshtein}, Integrable cluster dynamics of directed networks and
  pentagram maps, Adv. Math. \textbf{300} (2016), 390--450.

\bibitem{Schramm1998}
\textit{Z.-X. He} and \textit{O.~Schramm}, The {$C^\infty$}-convergence of
  hexagonal disk packings to the {R}iemann map, Acta Math. \textbf{180} (1998),
  no.~2, 219--245.

\bibitem{Hertrich-Jeromin2000}
\textit{U.~Hertrich-Jeromin}, Transformations of discrete isothermic nets and
  discrete cmc-1 surfaces in hyperbolic space, Manuscripta Math. \textbf{102}
  (2000), no.~4, 465--486.

\bibitem{Kojima2003}
\textit{S.~Kojima}, \textit{S.~Mizushima} and \textit{S.~P. Tan}, Circle
  packings on surfaces with projective structures, J. Differential Geom.
  \textbf{63} (2003), no.~3, 349--397.

\bibitem{Lam2018}
\textit{W.~Y. Lam}, Discrete minimal surfaces: critical points of the area
  functional from integrable systems, Int. Math. Res. Not. IMRN  (2018), no.~6,
  1808--1845.

\bibitem{Lam2017}
\textit{W.~Y. Lam}, Minimal surfaces from infinitesimal deformations of circle
  packings, Adv. Math. \textbf{362} (2020), 106939.

\bibitem{Lam2019}
\textit{W.~Y. Lam}, Quadratic differentials and circle patterns on complex
  projective tori, Geom. Topol. \textbf{25} (2021), no.~2, 961--997.

\bibitem{Lam2015a}
\textit{W.~Y. Lam} and \textit{U.~Pinkall}, Holomorphic vector fields and
  quadratic differentials on planar triangular meshes, in: Advances in discrete
  differential geometry, Springer, [Berlin], 2016, 241--265.

\bibitem{Lam2015}
\textit{W.~Y. Lam} and \textit{U.~Pinkall}, Isothermic triangulated surfaces,
  Math. Ann. \textbf{368} (2017), no. 1-2, 165--195.

\bibitem{Luo2004}
\textit{F.~Luo}, Combinatorial {Y}amabe flow on surfaces, Commun. Contemp.
  Math. \textbf{6} (2004), no.~5, 765--780.

\bibitem{Rodin1987}
\textit{B.~Rodin} and \textit{D.~Sullivan}, The convergence of circle packings
  to the {R}iemann mapping, J. Differential Geom. \textbf{26} (1987), no.~2,
  349--360.

\bibitem{Small1994}
\textit{A.~J. Small}, Surfaces of constant mean curvature {$1$} in {${\bf
  H}^3$} and algebraic curves on a quadric, Proc. Amer. Math. Soc. \textbf{122}
  (1994), no.~4, 1211--1220.

\bibitem{Stephenson2005}
\textit{K.~Stephenson}, Introduction to circle packing: The theory of discrete
  analytic functions, Cambridge University Press, Cambridge, 2005.

\bibitem{Thurston1986}
\textit{W.~P. Thurston}, Zippers and univalent functions, in: The {B}ieberbach
  conjecture ({W}est {L}afayette, {I}nd., 1985), Amer. Math. Soc., Providence,
  RI, 1986, \textit{Math. Surveys Monogr.}, volume~21, 185--197.

\bibitem{Umehara1993}
\textit{M.~Umehara} and \textit{K.~Yamada}, Complete surfaces of constant mean
  curvature {$1$} in the hyperbolic {$3$}-space, Ann. of Math. (2) \textbf{137}
  (1993), no.~3, 611--638.

\end{thebibliography}

\end{document}